\documentclass{amsart}%
\usepackage{palatino, mathpazo}
\usepackage{amsfonts}
\usepackage{amsmath}
\usepackage{amssymb,latexsym,xcolor}
\usepackage{graphicx}
\usepackage{harpoon}
\usepackage[mathscr]{eucal}
\usepackage{amssymb}%
\usepackage[
linktocpage=true,colorlinks,citecolor=magenta,linkcolor=blue,urlcolor=magenta]{hyperref}
\providecommand{\U}[1]{\protect \rule{.1in}{.1in}}

\newtheorem{theorem}{Theorem}[section]

\newtheorem{corollary}[theorem]{Corollary}

\newtheorem{Lemma}[theorem]{Lemma}

\newtheorem{Theorem}{Theorem}

\theoremstyle{remark}
\newtheorem{Remark}[theorem]{Remark}

\numberwithin{equation}{section}

\begin{document}
\title[Green functions on tori]{Green functions, Hitchin's formula and curvature equations on tori}
\author{Zhijie Chen}
\address{Department of Mathematical Sciences, Yau Mathematical Sciences Center,
Tsinghua University, Beijing, 100084, China }
\email{zjchen2016@tsinghua.edu.cn}
\author{Erjuan Fu}
\address{Beijing Institute of Mathematical Sciences and Applications, Beijing, 101408, China}
\email{fej.2010@tsinghua.org.cn, ejfu@bimsa.cn}
\author{Chang-Shou Lin}
\address{Department of Mathematics, National Taiwan University, Taipei 10617, Taiwan }
\email{cslin@math.ntu.edu.tw}

\begin{abstract}
Let $G(z)=G(z;\tau)$ be the Green function on the flat torus $E_{\tau}=\mathbb{C}/(\mathbb{Z}+\mathbb{Z}\tau)$ with the singularity at $0$. 
Lin and Wang (Ann. Math. 2010) proved that $G(z)$ has either $3$ or $5$ critical points (depending on the choice of $\tau$). Later, Bergweiler and Eremenko (Proc. Amer. Math. Soc. 2016) gave a new proof of this remarkable result by using anti-holomorphic dynamics.

In this paper, firstly, we prove that once  $G(z)$ has $5$ critical points, then these $5$ critical points are all non-degenerate. Secondly, we study the sum of two Green functions which can be reduced to $G_p(z):=\frac12(G(z+p)+G(z-p))$. We prove that for any $p$ satisfying $p\neq -p$ in $E_{\tau}$, the number of critical points of $G_p(z)$ belongs to $\{4,6,8,10\}$ (depending on the choice of $(\tau, p)$) and each number really occurs. We apply Hitchin's formula (J. Differ. Geom. 1995) in a surprising way to prove the generic non-degeneracy of critical points. This allows us to study the distribution of the numbers of critical points of $G_p(z)$ as $p$ varies. Applications to
 the curvature equation
$\Delta u+e^{u}=4\pi(\delta_{p}+\delta_{-p})$ on $E_{\tau}$ are also given, and how the geometry of the torus affects the solution structure is studied.
\end{abstract}


\maketitle

\section{Introduction}

Let $\tau \in \mathbb{H}=\left \{  \tau\in\mathbb C|\operatorname{Im}\tau>0\right \}$, $\Lambda_{\tau}=\mathbb{Z}+\mathbb{Z}\tau$, and denote
$$\omega_{0}=0,\quad\omega_{1}=1,\quad\omega_{2}=\tau,\quad\omega_{3}=1+\tau.$$Let $E_{\tau}:=\mathbb{C}/\Lambda_{\tau}$ be a flat torus in the
plane and $E_{\tau}[2]:=\{ \frac{\omega_{k}}{2}|k=0,1,2,3\}+\Lambda
_{\tau}$ be the set consisting of the lattice points and half periods
in $E_{\tau}$.  

In this paper, we study the Green function of the flat torus.
One of our motivations for this research is to understand how the geometry of 
the torus affects the solution structure of elliptic PDEs. Consider the curvature equation
\begin{equation}
\Delta u+e^{u}=4\pi(\delta_{p}+\delta_{-p})\quad\text{ on
}\; E_{\tau}, \label{mean}%
\end{equation}
where $\delta_p$ denotes the Dirac measure at $p$, 
we will prove that

\begin{theorem}\label{main-thm-10}
Let $G(z)=G(z;\tau)$ be the Green function of the torus $E_\tau$ and let $p\in E_{\tau}\setminus E_{\tau}[2]$. 
\begin{itemize}
\item[(1)]
There is a sequence of even solutions $u_p(z)$ of \eqref{mean} such that $u_p(z)$ blows up as $0\neq p\to 0$ if and only if $G(z)$ has exactly $3$ critical points and one of them is degenerate.
\item[(2)] There is a sequence of even solutions $u_p(z)$ of \eqref{mean} such that $u_p(z)$ converges to a solution of \begin{equation}\label{mfe-8pi}\Delta u+e^u=8\pi \delta_0 \quad\text{on }E_{\tau}\end{equation} as $0\neq p\to 0$ if and only if $G(z)$ has exactly $5$ critical points.
\item[(3)] There is small $\varepsilon>0$ such that \eqref{mean} has no solutions for any $0<|p|<\varepsilon$ if and only if $G(z)$ has exactly $3$ critical points that are all non-degenerate.
\end{itemize}
\end{theorem}

We will explain in Section \ref{section6} that the assumption ``even'' is necessary for the validity of Theorem \ref{main-thm-10} (1)-(2).

\subsection{The Green function $G$} 
The Green function $G(z,w)=G(z,w;\tau)$ of the flat torus $E_{\tau}$ is the unique function that satisfies
\[
-\Delta_z G(z, w)=\delta_{w}-\frac{1}{\left \vert E_{\tau}\right \vert }\text{
\ on }E_{\tau},\quad
\int_{E_{\tau}}G(z,w)dxdy=0,
\]
where we use the complex variable $z=x+iy$, $\Delta_z=\frac{\partial^2}{\partial x^2}+\frac{\partial^2}{\partial y^2}=4\partial^2_{\bar z z}$ is the Laplace operator, $\delta_{w}$ is the Dirac measure at $w$ and $\vert E_{\tau}\vert$ is
the area of the torus $E_{\tau}$. By the translation invariance of $\Delta_z$, we have $G(z,w)=G(z-w,0)$ and it is enough to consider the Green function
$G(z;\tau)=G(z):=G(z,0)$.
Clearly $G(z)$ is an even function on $E_{\tau}$ with the only
singularity at $0$, so $-a$ is also a critical point of $G(z)$ if $a\in E_{\tau}\setminus\{0\}$ is. 
For $k\in \{1,2,3\}$, since
\begin{align}\label{gde}\nabla G\Big(\frac{\omega_k}{2}\Big)=\nabla G\Big(-\frac{\omega_k}{2}\Big)=-\nabla G\Big(\frac{\omega_k}{2}\Big)=0,\end{align}
we see that $\frac{\omega_k}{2}$, $k\in \{1,2,3\}$, are always critical points of $G(z)$. 

\medskip

\noindent{\bf Definition.} {\it A critical point $a\in E_{\tau}$ of $G$ (resp. $G_p$; see below) is called trivial if $a=-a$ in $E_{\tau}$, i.e. $a\in E_{\tau}[2]$.  A critical point $a\in E_{\tau}$ is called nontrivial if $a\neq-a$ in $E_{\tau}$, i.e. $a\notin E_{\tau}[2]$.}
\medskip

Therefore, nontrivial critical points, if exist, must appear in pairs.
 Lin and Wang \cite{LW} studied the number of critical points of $G(z)$, and proved the following remarkable result.
 
\begin{Theorem} \cite{LW} \label{thm-LW} $G(z;\tau)$ has at most one pair of nontrivial critical points, or equivalently, $G(z;\tau)$ has either $3$ or $5$ critical points (depends on the choice of $\tau$).
 For example, 
\begin{itemize}
\item when $\tau \in i\mathbb{R}_{>0}$, i.e. $\tau$ is purely imaginary and so $E_{\tau}$ is a rectangular torus,  $G(z;\tau)$ has exactly $3$ critical points  $\frac{\omega_k}{2}$, $k\in \{1,2,3\}$, which are all non-degenerate;
\item while for $\tau=\frac{1}{2}+\frac{\sqrt{3}}{2}i$, i.e. $E_{\tau}$ is a rhombus torus, $G(z;\tau)$ has exactly $5$ critical points, or equivalently, has a unique pair of nontrivial critical points $\pm\frac{\omega_3}{3}$. 
\end{itemize}
 \end{Theorem}
 
Lin-Wang \cite{LW} proved Theorem \ref{thm-LW} by showing that there is a one-to-one correspondence between pairs of nontrivial critical points of $G(z;\tau)$ and even solutions of the curvature equation \eqref{mfe-8pi}.
Then by proving that \eqref{mfe-8pi} has at most one even solution by PDE methods, they finally obtained that $G(z;\tau)$ has at most one pair of nontrivial critical points. See \cite{LW} for details.
Later, Lin-Wang \cite{LW4} proved the following result.

\begin{Theorem}\cite{LW4}\label{thm-A0}
Suppose that the pair of nontrivial critical points $\pm q$ of $G(z)$ exists, then $\pm q$ are the minimal points of $G(z)$, and all the three trivial critical points $\frac{\omega_k}{2}$ are non-degenerate saddle points of $G(z)$, that is, the Hessian $\det D^2G(\frac{\omega_k}{2})<0$ for $k=1,2,3$.
\end{Theorem}

In 2016, Bergweiler and Eremenko \cite{BE} gave a new proof of Theorem \ref{thm-LW} by using anti-holomorphic dynamics. We will briefly review their proof in Section \ref{section2}, and prove that a further discussion of their proof together with Theorem \ref{thm-A0} actually implies the following result.

\begin{theorem}\label{thm-A00}
Suppose that the pair of nontrivial critical points $\pm q$ of $G(z)$ exists, then $\pm q$ are also non-degenerate, i.e. $\det D^2G(\pm q)>0$.
\end{theorem}

A consequence of Theorems \ref{thm-A0} and \ref{thm-A00} is that if $G(z)$ has a degenerate critical point, then $G(z)$ has only trivial critical points.
It is well known that the non-degeneray of critical points of Green functions has many important applications, such as in constructing bubbling solutions via the reduction method and proving the local uniqueness of bubbling solutions for elliptic PDEs; see e.g. \cite{BKLY, CL-2, LW4, LY} and references therein. We believe that Theorem \ref{thm-A00} will have important applications to other problems.

\subsection{The Green function $G_p$}
The study of \eqref{mean} and the proof of Theorem \ref{main-thm-10} lead us to study the sum of two Green fuctions $G(z-p_1)+G(z-p_2)$. By changing variable $z\mapsto z+\frac{p_1+p_2}{2}$, we can always assume $p_2=-p_1$, so it is enough to study the Green function
 $G_p(z)=G_{-p}(z)$, which is defined by
\begin{equation}G_p(z)=G_p(z;\tau):=\frac12\big(G(z-p;\tau)+G(z+p;\tau)\big).\end{equation}
Remark that if $p\in E_{\tau}[2]$, i.e. $p=\frac{\omega_k}{2}$ for some $k$, then $G_p(z)=G(z-p)$ and so Theorem \ref{thm-LW} implies that $G_p(z)$ has either $3$ or $5$ critical points. 

In this paper, we consider the case $p\in E_{\tau}\setminus E_{\tau}[2]$. Clearly $G_p(z)$ is also even, so the same argument as \eqref{gde} implies that $\frac{\omega_k}{2}$, $k=0,1,2,3$, are all trivial critical points of $G_p(z)$, i.e. the number of critical points of $G_p(z)$ is an even number of at least $4$. 
Here we generalize Theorem \ref{thm-LW} to $G_p(z)$ as follows.

\begin{theorem}\label{main-thm-1} Let $p\in E_{\tau}\setminus E_{\tau}[2]$. Then $G_p(z)$ has at most $3$ pairs of nontrivial critical points, or equivalently, the number of critical points of $G_p(z)$ belongs to $\{4,6,8,10\}$.

Moreover, if $\wp(2p)=2\eta_1-\frac{2\pi}{\operatorname{Im}\tau}$, then $G_p(z)$ has at most one pair of nontrivial critical points, or equivalently, the number of critical points of $G_p(z)$ belongs to $\{4,6\}$.
\end{theorem}

Here $\wp(z)$ and $\eta_1=\eta_1(\tau)$ are classical special functions that will be recalled soon.
In Section \ref{section2}, we will follow the approach of using anti-holomorphic dynamics from Bergweiler-Eremenko \cite{BE} to prove Theorem \ref{main-thm-1}. After Theorem \ref{main-thm-1}, a natural question is whether there exists $(\tau, p)$ such that $G_p(z)=G_p(z;\tau)$ has exactly $N$ critical points for every $N\in \{4,6,8,10\}$. This question can not be settled from the proof of Theorem \ref{main-thm-1}.
Recently, it was proved in \cite{CKL-2025} that the lower bound $4$ occurs; we will recall this result in Theorem \ref{thm-B}. By developing new ideas,
our next result confirms that every $N\in \{6, 8, 10\}$ also occurs. 

\begin{theorem}\label{main-thm-2} Fix $k\in\{1,2,3\}$.
\begin{itemize}
\item[(1)]There exist $\tau$ and $\varepsilon>0$ small such that if $|p-\frac{\omega_k}{4}|<\varepsilon$, $G_p(z;\tau)$ has exactly $10$ critical points, which are all non-degenerate with $\det D^2G_p(\frac{\omega_l}{2};\tau)<0$ for all $l\in\{0,1,2,3\}$.
\item[(2)] There exists $\tau$ such that $G_{\frac{\omega_k}{4}}(z;\tau)$ has exactly $6$ critical points with $\det D^2G_{\frac{\omega_k}{4}}$ $(\frac{\omega_l}{2};\tau)<0$ for all $l\in\{0,1,2,3\}$, but
the unique pair of nontrivial critical points are degenerate minimal points of $G_{\frac{\omega_k}{4}}(z;\tau)$.
\item[(3)] There exists $\tau$ such that $G_{\frac{\omega_k}{4}}(z;\tau)$ has exactly $6$ critical points, among which there are exactly $2$ trivial critical points being degenerate, but the other $4$ critical points (including the pair of nontrivial critical points) are all non-degenerate saddle points of $G_{\frac{\omega_k}{4}}(z;\tau)$. Furthermore,
\begin{itemize}
\item[(3-1)] There exists $p$ close to $\frac{\omega_k}{4}$ such that $G_{p}(z;\tau)$ has exactly $8$ critical points, which are all non-degenerate.
\item[(3-2)] There also exists $p$ close to $\frac{\omega_k}{4}$ such that $G_{p}(z;\tau)$ has exactly $6$ or $10$ critical points, which are all non-degenerate.
\end{itemize}
\end{itemize}
\end{theorem}

\begin{Remark}
We compare Theorems \ref{thm-A0}-\ref{thm-A00} with Theorem \ref{main-thm-2}. 
Theorems \ref{thm-A0}-\ref{thm-A00} say that the unique pair of nontrivial critical points of $G(z)$ must be non-degenerate minimal points. However, this is not the case for $G_p(z)$. When $G_p(z)$ has $6$ critical points, or equivalently, has a unique pair of nontrivial critical points $\pm q$, then $\pm q$ could be degenerate minimal points of $G_p(z)$ for some $(\tau, p)$, and also could be saddle points of  $G_p(z)$ for some other $(\tau, p)$. 
This phenomenon indicates that the structure of critical points of $G_p(z)$ is much more complicated than that of $G(z)$.
\end{Remark}


An interesting question is that given $q\in E_{\tau}\setminus E_{\tau}[2]$, how many $\pm p\in E_{\tau}\setminus E_{\tau}[2]$ might there be such that $q$ is a nontrivial critical point of $G_p$? To settle this question, we recall from \cite{LW} that
the Green function $G(z)$ can be expressed explicitly in terms of elliptic
functions. Let $\wp(z)=\wp( z;\tau)$ be the Weierstrass $\wp$ function with periods
$\Lambda_{\tau}$, defined by%
\[\wp(z;\tau):=\frac{1}{z^{2}}+\sum_{\omega \in \Lambda_{\tau
}\backslash\{0\}  }\left(  \frac{1}{(z-\omega)^{2}}-\frac
{1}{\omega^{2}}\right) ,
\]
which satisfies the well-known cubic equation
\[\wp^{\prime}(z;\tau)^{2}=4\wp(z;\tau)^{3}-g_{2}(\tau)\wp
(z;\tau)-g_{3}(\tau)=4\prod_{k=1}^3(\wp(z;\tau)-e_k(\tau)),
\]
where $g_2, g_3$ are known as invariants of the elliptic curve, and $e_k=e_k(\tau):=\wp(\frac{\omega_k}{2};\tau)$ for $k=1,2,3$. 
Let $\zeta(z)=\zeta(z;\tau):=-\int^{z}\wp(\xi;\tau)d\xi$ be the Weierstrass
zeta function with two quasi-periods
\begin{equation}
\eta_k=\eta_{k}(\tau):=2\zeta\Big(\frac{\omega_{k}}{2};\tau\Big)=\zeta(z+\omega_{k}%
;\tau)-\zeta(z;\tau),\quad k=1,2.\label{40-2}%
\end{equation}
This $\zeta(z)$ is an odd meromorphic function with simple poles at $\Lambda_{\tau}$.
Then by using the complex variable $z\in\mathbb C$, it was proved in \cite{LW} that
\begin{equation}
\label{G_z}-4\pi \frac{\partial G}{\partial z}(z)=\zeta(z)-r\eta_{1}
-s\eta_{2},
\end{equation}
where $(r,s)$ is defined by $z=r+s\tau$ with $r,s\in \mathbb{R}.$
Note that $z\notin E_{\tau}[2]$ is equivalent to $(r,s)\in\mathbb{R}^2\setminus\frac12\mathbb{Z}^2$, where $\frac12\mathbb{Z}^2:=\{(r,s)\,:\, 2r,2s\in\mathbb Z\}.$
By \eqref{G_z}, we see that $q=r+s\tau$ with $(r,s)\in\mathbb{R}^2\setminus\frac12\mathbb{Z}^2$ is a nontrivial critical point of $G_p$ if and only if 
\begin{equation}\label{a+001p}\zeta(q+p)+\zeta(q-p)-2(r\eta_1+s\eta_2)=0.\end{equation}
Using the additional formula of elliptic functions
$$\zeta(z+w)+\zeta(z-w)-2\zeta(z)=\frac{\wp'(z)}{\wp(z)-\wp(w)},$$
it is easy to see that \eqref{a+001p} is equivalent to
\begin{align}\label{513-1} 
\wp(p)=\wp (r+s\tau)+\frac{\wp ^{\prime }(r+s\tau)}{%
2(\zeta(r+s\tau)-r\eta_1-s\eta_2)}.
\end{align}
It is well known that $\wp(\cdot): E_{\tau}\to \mathbb{C}\cup\{\infty\}$ is a double cover with branch points at $\frac{\omega_k}{2}$'s, i.e. for any $c\in \mathbb{C}\cup\{\infty\}$, there is a unique pair $\pm z_c\in E_{\tau}$ such that $\wp(\pm z_{c})=c$. Consequently, for any $(r,s)\in \mathbb{R}^{2}\backslash \frac{1}{2}\mathbb{Z}^{2}$, there exists a unique pair $\pm p=\pm p_{r,s}(\tau)\in E_{\tau}$ such that
\eqref{513-1} holds. Note that $p_{r,s}(\tau)\in E_{\tau}[2]$ might hold for some $(r,s)\in \mathbb{R}^{2}\backslash \frac{1}{2}\mathbb{Z}^{2}$.
Therefore, \eqref{513-1} implies that there is at most one pair $\pm p\in E_{\tau}\setminus E_{\tau}[2]$ such that $q=r+s\tau$ is a nontrivial critical point of $G_p$. This argument also implies the following result.

\begin{theorem}\label{thm-0B}
Let $p\in E_{\tau}\setminus E_{\tau}[2]$ and $(r,s)\in \mathbb{R}^{2}\backslash \frac{1}{2}\mathbb{Z}^{2}$. Then $\pm(r+s\tau)$ is a pair of nontrivial critical points of $G_p(z)$ if and only if \eqref{513-1} holds.
\end{theorem}

The \eqref{513-1} is well known as Hitchin's formula \cite{Hit1}, where Hitchin studied Einstein metrics and proved that for any fixed $(r,s)\in \mathbb{C}^{2}\setminus \frac{1}{2}\mathbb{Z}^{2}$, the $p_{r,s}(\tau)$ defined by \eqref{513-1}, as a function of $\tau\in \mathbb{H}$, is a solution of the following elliptic form of a certain Painlev\'{e} VI equation
\begin{equation}\label{PVI}
\frac{d^{2}p(\tau)}{d\tau^{2}}=\frac{-1}{32\pi^{2}}\sum_{k=0}^{3}
\wp^{\prime}\left( p(\tau)+\frac{\omega_{k}}{2};
\tau \right).\end{equation}
Furthermore, those solutions $p_{r,s}(\tau)$ with one of $(r,s)$ real and the other one purely imaginary have important applications to Einstein metrics.
See e.g. \cite{CKL-PAMQ,CKL-2025,CKLW,Hit1} and references therein for introductions of Painlev\'{e} VI equations. In particular, it was introduced in \cite{CKL-PAMQ,CKL-2025} that \eqref{PVI} governs the isomonodromic deformation of a second order linear ODE (see \eqref{GLE} in Section \ref{section6}). Furthermore, it was proved in \cite[Theorem 4.1]{CKL-PAMQ} that up to a common conjugation, the basic monodromy matrices $N_1$ and $N_2$ (see \eqref{monon1n2} in Section \ref{section6}) of the ODE \eqref{GLE} associated to the solution $p_{r,s}(\tau)$ of \eqref{PVI} can be expressed as
$$N_1=\begin{pmatrix}
e^{-2\pi is}&\\
&e^{2\pi is}
\end{pmatrix},\quad N_2=\begin{pmatrix}
e^{2\pi ir}&\\
&e^{-2\pi ir}
\end{pmatrix},$$
which are both unitary matrices if we further assume $(r,s)\in \mathbb{R}^{2}\setminus \frac{1}{2}\mathbb{Z}^{2}$.
Therefore, Theorem \ref{main-thm-1} is equivalent to the following result.

\begin{theorem}\label{main-thm-PVI}
Fix $\tau_0\in\mathbb{H}$ and $p_0\in E_{\tau_0}\setminus E_{\tau_0}[2]$. Then the Painlev\'{e} VI equation \eqref{PVI} has at most $3$ solutions $p(\tau)$ satisfying $p(\tau_0)=p_0$ such that all the monodromy matrices of the associated second order linear ODE \eqref{GLE} are unitary matrices.
\end{theorem}
\begin{proof}
Theorem \ref{main-thm-PVI} follows directly from Theorem \ref{thm-0B}, Theorem \ref{main-thm-1} and \cite[Theorem 4.1]{CKL-PAMQ}. 
\end{proof}

Theorem \ref{main-thm-PVI} is highly nontrivial from the viewpoint of Painlev\'{e} VI equations. Indeed, for fixed $\tau$ and $p\in E_{\tau}\setminus E_{\tau}[2]$, we consider the set
\begin{equation}\label{sp}\mathcal{S}(\wp(p)):=\Big\{(r,s)\in\mathbb{C}^2 \;\Big|\; r+s\tau\notin E_{\tau}[2]\;\text{ and \eqref{513-1} holds}\Big\}.\end{equation}
Note that $(-r,-s)\in \mathcal{S}(\wp(p))$ if $(r,s)\in \mathcal{S}(\wp(p))$.
Then Theorem \ref{main-thm-PVI} says that $\#((\mathcal{S}(\wp(p))\cap\mathbb R^2)/\mathbb{Z}^2)\leq 6$. In general, one might ask whether there is an upper bound for $\#((\mathcal{S}(\wp(p))\cap ((t_1,t_2)+\mathbb R^2))/\mathbb{Z}^2)$ for given $t_1,t_2\in\mathbb{C}$. We will prove in Remark \ref{section2-2} that the upper bound is $10$ for any $t_1,t_2\in\mathbb{C}$.

Conversely, we apply Hitchin's formula \eqref{513-1} in a quite different way to prove the non-degeneracy of all critical points for generic $p$.

\begin{theorem}\label{main-thm-6}
Fix any $\tau$. Then for almost all $p\in E_{\tau}\setminus E_{\tau}[2]$, all critical points of $G_p(z)$ are non-degenerate.
\end{theorem}

Theorem \ref{main-thm-6} is extremely useful for us to count the number of critical points. To do this, the Brouwer degree theory plays an important role. For any non-degenerate critical point $q$ of $G_p(z)$, it is well known that the local degree (denoted by $\deg_p(q)$) of $\nabla G_p(z)$ at $z=q$ is $1$ (resp. $-1$) if $\det D^2G_p(q)>0$ (resp. $\det D^2G_p(q)<0$). Suppose that all critical points of $G_p(z)$ are non-degenerate, then we will see in Remark \ref{rmk3-11} that the degree counting formula is
\begin{equation}\label{deg-count}
\sum_{\text{$q$ is a critical point of $G_p$}}\deg_p(q)=-2.
\end{equation}
Therefore, the information whether the trivial critical points $\frac{\omega_k}{2}$ are degenerate or not is important.
For this purpose, we define
\begin{equation}\label{B00}
\mathcal{B}_0:=\Big\{z\in\mathbb{C}\; :\; \Big|z-\Big(\frac{\pi}{\operatorname{Im}\tau}-\eta_1\Big)\Big|<\frac{\pi}{\operatorname{Im}\tau}\Big\},
\end{equation}
and for $k\in\{1,2,3\}$, we define
\begin{align}\label{alphak0}
\alpha_k:=\frac{\frac{\pi}{\operatorname{Im}\tau}-(\eta_1+e_k)}{3e_k^2-\frac{g_2}{4}},\quad \beta_k:=\frac{\pi}{|3e_k^2-\frac{g_2}{4}|\operatorname{Im}\tau}>0,
\end{align}
\begin{equation}\label{alphak1}
\mathcal{B}_k:=\begin{cases}\bigg\{z\in\mathbb{C}\; :\; \bigg|z-e_k-\frac{\overline{\alpha_k}}{|\alpha_k|^2-\beta_k^2}\bigg|<\frac{\beta_k}{\left||\alpha_k|^2-\beta_k^2\right|}\bigg\}\quad\text{if }|\alpha_k|\neq \beta_k,\\
\Big\{z\in\mathbb{C}\; :\; \operatorname{Re}(\alpha_k (z-e_k))>\frac12\Big\}\quad\text{if }|\alpha_k|=\beta_k,\end{cases}
\end{equation}
that is, $\mathcal{B}_k$ is either an open disk or an open half plane. We will prove in Section \ref{section3} that $|\alpha_k|=\beta_k$ is equivalent to that $\frac{\omega_k}{2}$ is a degenerate critical point of $G(z)$. It was proved in \cite[Theorem 1.2]{CKLW} (we will recall it in Theorem \ref{thm-CKLW}) that for any $\tau$, there is at most one trivial critical point of $G(z;\tau)$ being degenerate. This implies that there is at most one half plane among $\mathcal{B}_k$'s. For example, if $\tau=e^{\pi i/3}=\frac12+\frac{\sqrt{3}}{2}i$ or $\tau=i$, then $\mathcal{B}_k$'s are all disks, the figures of which are presented in Figures 1 and 2.
Our next result gives the criterion for the degeneracy of trivial critical points.

\begin{theorem}\label{main-thm-01} Let $p\in E_{\tau}\setminus E_{\tau}[2]$ and fix $k\in\{0,1,2,3\}$. Then $\frac{\omega_k}{2}$ is a degenerate critical point of $G_p(z)$ if and only if $\wp(p)\in \partial\mathcal{B}_k$, where $\mathcal{B}_k$ is defined by \eqref{B00}-\eqref{alphak1}.
\end{theorem}

\begin{figure}
\label{O5-1}\includegraphics[width=1.8in]{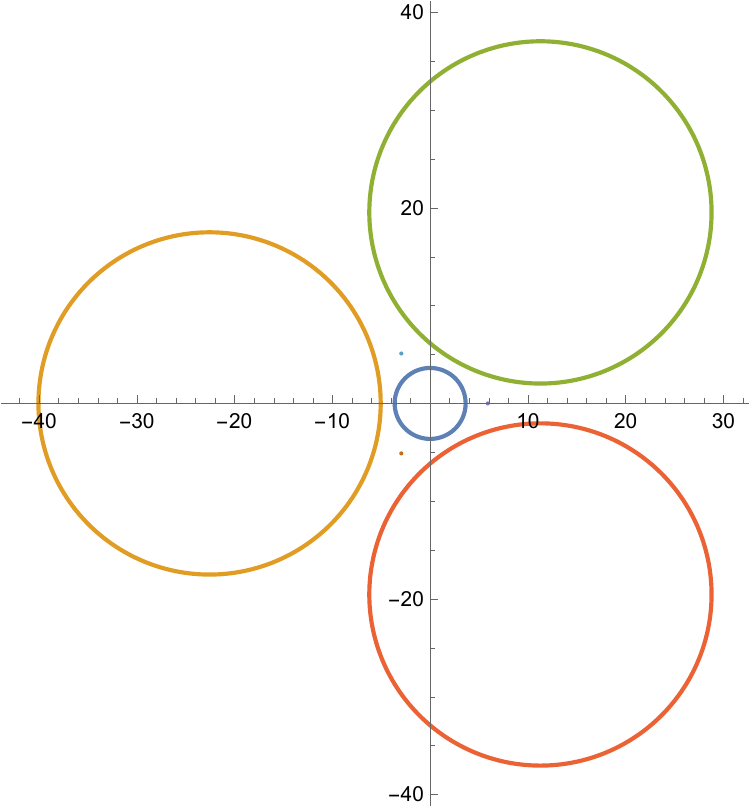}\caption{The four disks and $\{e_1,e_2,e_3\}$ for $\tau=e^{\pi i/3}=\frac12+\frac{\sqrt{3}}{2}i$: the smallest circle for $\partial\mathcal{B}_0$, left for $\partial\mathcal{B}_1$, upper right for $\partial\mathcal{B}_2$ and lower right for $\partial\mathcal{B}_3$.}%
\end{figure}

\begin{figure}
\label{O5-1}\includegraphics[width=1.8in]{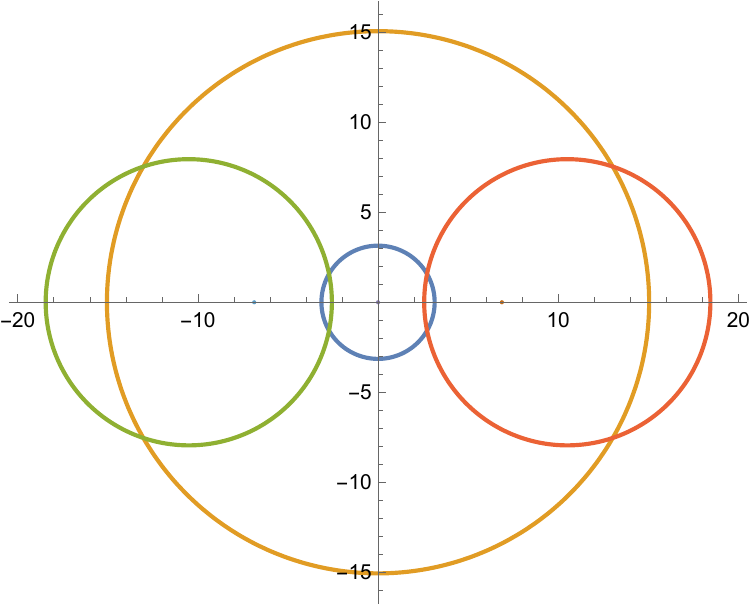}\caption{The four disks and $\{e_1,e_2,e_3\}$ for $\tau=i$: the smallest circle for $\partial\mathcal{B}_0$, biggest for $\partial\mathcal{B}_3$, left for $\partial\mathcal{B}_1$ and right for $\partial\mathcal{B}_2$.}%
\end{figure}

Thanks to Theorems \ref{main-thm-6}-\ref{main-thm-01}, we can determine more accurately the possible numbers of critical points of $G_p(z)$ for generic $p$ as follows.

\begin{theorem}
\label{thm-section5}
Fix $\tau$. Let $\Xi$ be a connected component of the open set $\mathbb{C}\setminus (\{e_1,e_2,e_3\}\cup\cup_{k=0}^3\partial \mathcal{B}_k)$, and define
$$m(\Xi):=\#\Big\{\frac{\omega_k}{2}\;: \;0\leq k\leq 3,\; \det D^2G_p\Big(\frac{\omega_k}{2}\Big)>0\;\text{for }p\in\wp^{-1}(\Xi)\Big\},$$ which is a constant independent of the choice of $p\in \wp^{-1}(\Xi)$ by Theorem \ref{main-thm-01}. Then $m(\Xi)\leq 2$, and
\begin{itemize}
\item[(1)] if $m(\Xi)=0$, then $G_p(z)$ has at least $6$ critical points for any $\wp(p)\in\Xi$, and $G_p(z)$ has exactly either $6$ or $10$ critical points that are all non-degenerate for almost all $\wp(p)\in \Xi$. Furthermore, if the number is $6$, then the unique pair of nontrivial critical points are the minimal points of $G_p(z)$.
\item[(2)] if $m(\Xi)=1$,  then $G_p(z)$ has exactly either $4$ or $8$ critical points that are all non-degenerate for almost all $\wp(p)\in \Xi$.
\item[(3)] if $m(\Xi)=2$, then $G_p(z)$ has exactly $6$ critical points for any $\wp(p)\in\Xi$, and the unique pair of nontrivial critical points are non-degenerate saddle points of $G_p(z)$ for almost all $\wp(p)\in \Xi$.
\end{itemize}
\end{theorem}

\begin{Remark} Theorem \ref{thm-section5}-(3) shows that the number of critical points is always $6$ for all $\wp(p)\in \Xi$ when $m(\Xi)=2$. This result is stronger than the degree counting formula \eqref{deg-count} because the possibility of $10$ critical points are excluded in Theorem \ref{thm-section5}-(3). We will see from Theorem \ref{main-thm-8}-(2) for the case $\tau=i$ that $m(\Xi)=2$ could occur. Theorem \ref{thm-section5} is almost sharp due to the following reasons.

\begin{itemize}
\item[(1)] Theorem \ref{main-thm-2}-(1) and Theorem \ref{thm-5c}  in Section \ref{section3} indicate that the numbers both $6$ and $10$ in Theorem \ref{thm-section5}-(1) could occur.  Theorem \ref{main-thm-2}-(3-1) and Theorem \ref{thm-B} in Section \ref{section3} indicate that the numbers both $4$ and $8$ in Theorem \ref{thm-section5}-(2) could occur.

\item[(2)] Since nontrivial critical points might be degenerate by Theorem \ref{main-thm-2}-(2), we can not expect that the number of critical points is the same in $\Xi$ for all $\Xi$'s. For example, for $\Xi$ with $m(\Xi)=0$, $\Xi$ might contain some curves along which $G_p(z)$ has degenerate nontrivial critical points, and then the number of critical points should change when $\wp(p)\in \Xi$ crosses these curves. In other words, there may happen that the number of critical points is $6$ for $\wp(p)$ in some open subset of $\Xi$ while the number becomes $10$ for $\wp(p)$ in some other open subset of $\Xi$. Our next result confirms that this phenomenon really occurs.
\end{itemize}
\end{Remark}

\begin{theorem}
\label{thm-section5-5}
There exist $\tau$ and a connected component $\Xi$ of $\mathbb{C}\setminus (\{e_1,e_2,e_3\}\cup\cup_{k=0}^3\partial \mathcal{B}_k)$ satisfying $m(\Xi)=0$ such that the following holds: Define 
\[\Omega_N:=\bigg\{\wp(p)\in\Xi\,:\begin{array}{l}\text{$G_p(z;\tau)$ has exactly $N$ critical points}\\\text{that are all non-degenerate}\end{array}\bigg\}\;\text{for }N=6,10,\]
then $\Omega_6\neq \emptyset$ and $\Omega_{10}\neq \emptyset$ are both open subsets of $\Xi$,  $\partial\Omega_6\cap\partial\Omega_{10}\cap\Xi\neq\emptyset$, and $G_p(z;\tau)$ has degenerate nontrivial critical points for any $\wp(p)\in(\partial\Omega_6\cup\partial\Omega_{10})\cap\Xi$. 
\end{theorem}

Clearly the number of critical points changes when $\wp(p)\in\Xi$ crosses $\partial\Omega_6\cap\partial\Omega_{10}$.
Theorem \ref{thm-section5-5} shows that Theorem \ref{thm-section5}-(1) is sharp.
Theorems \ref{thm-section5} and \ref{thm-section5-5} will be proved in Section \ref{section5}, where we will also apply Theorem \ref{thm-section5} to two special cases $\tau=\frac12+\frac{\sqrt{3}}{2}i$ and $\tau=i$; See Theorems \ref{main-thm-08} and \ref{main-thm-8} for details. 
In subsequet work, we will continue to study $G_p(z)$ for $\tau=ib$ or $\tau=\frac12+ib$ with $b>0$, i.e. when $E_{\tau}$ is a general rectangular torus or rhombus torus. 



Finally, concerning the degeneracy of nontrivial critical points, Theorem \ref{main-thm-2}-(2) and Theorem \ref{thm-section5-5} indicate that nontrivial critical points could be degenerate for some $(\tau, p)$. 
Here
we propose the following conjecture.

\medskip
\noindent{\bf Conjecture A.} {\it If $G_p(z;\tau)$ has $10$ critical points, then these critical points are all non-degenerate.}
\medskip

Unfortunately, the idea of proving Theorem \ref{thm-A00} does not work for Conjecture A. We should study this conjecture in future.

\subsection{The curvature equation}
As applications, we study the following curvature equation with two singular sources
\begin{equation}
\Delta u+e^{u}=4\pi(\delta_{p_1}+\delta_{p_2})\quad\text{ on
}\; E_{\tau}. \label{0mean}%
\end{equation}
By changing variable $z\mapsto z+\frac{p_1+p_2}{2}$, we can always assume $p_2=-p_1$, namely we only need to study \eqref{mean}:
\begin{equation*}
\Delta u+e^{u}=4\pi(\delta_{p}+\delta_{-p})\quad\text{ on
}\; E_{\tau}. 
\end{equation*}

Equation (\ref{mean}) arises from conformal geometry and mathematical physics. Geometrically, a solution $u$ of (\ref{mean}) leads to a metric
$ds^{2}=\frac{1}{2}e^{u}dz^2$ with constant Gaussian curvature $+1$
acquiring conic singularities at $\pm p$. It also
appears in statistical physics as the equation for the mean field
limit of the Euler flow in Onsager's vortex model (cf. \cite{CLMP}), hence
also called a mean field equation. 
We refer the readers to \cite{BKLY,CLW,CL-1,CL-2,CL-3,CL-AJM,EG-2015,EG,LW,LW2,LY,MR,NT1} and references therein for
recent developments on mean filed equations.

Remark that if $p\in E_{\tau}[2]$, i.e. $p=-p$ in $E_{\tau}$, then by changing variable $z\mapsto z+p$, \eqref{mean} becomes the curvature equation \eqref{mfe-8pi}, which has been studied in \cite{LW}.
 Therefore, in this paper we consider the case $p\in E_{\tau}\setminus E_{\tau}[2]$. Inspired by the study of \eqref{mfe-8pi} in \cite{LW} , Chen-Kuo-Lin \cite{CKL-2025} showed that the existence of solutions of \eqref{mean} depends on the location of the singularity $p$, as stated as follows.

\begin{Theorem}\cite{CKL-2025}\label{thm-C}
Let $p\in E_{\tau}\setminus E_{\tau}[2]$. Then the following statements hold.
\begin{itemize}
\item[(1)] Once \eqref{mean} has a solution, then it has a one-parameter scaling family of solutions $u_{\beta}(z)$, where $\beta>0$ is arbitrary. Furthermore, $u_\beta(z)$ is even (i.e. $u_{\beta}(z)=u_\beta(-z)$) if and only if $\beta=1$.
\item[(2)] \eqref{mean} has solutions if and only if $G_p(z)$ has nontrivial critical points.
\end{itemize}
\end{Theorem}

In this paper, we follow the approach from \cite{CLW,LW} to prove the following result.

\begin{theorem}\label{main-thm-3}
Let $p\in E_{\tau}\setminus E_{\tau}[2]$. Then there is a one-to-one correspondence between pairs of nontrivial critical points of $G_p(z)$ and one-parameter scaling families of solutions (or equivalently, even solutions) of \eqref{mean}. 

Consequently, \eqref{mean} has at most $3$ one-parameter scaling families of solutions, and every number in $\{0,1,2,3\}$ really occurs. 
\end{theorem}

The rest of this paper is organized as follows. In Section \ref{section2}, we prove Theorems \ref{thm-A00} and \ref{main-thm-1}. In Section \ref{section3}, we prove Theorems \ref{main-thm-2} and \ref{main-thm-01}. In Section \ref{section4}, we prove  Theorem \ref{main-thm-6}. In Section \ref{section5}, we prove Theorems \ref{thm-section5} and \ref{thm-section5-5}. Finally, Theorems \ref{main-thm-10} and \ref{main-thm-3} will be proved in Section \ref{section6}.

\section{Possible numbers of critical points}

\label{section2}

In this section, we follow the approach of using anti-holomorphic dynamics from \cite{BE} to prove Theorems \ref{thm-A00} and \ref{main-thm-1}.

Fix any $p\in E_{\tau}\setminus E_{\tau}[2]$. 
Recall \eqref{a+001p} that $z=r+s\tau$ with $r,s\in\mathbb R$ is a critical point of $G_p$ if and only if 
\begin{equation}\label{a+1p}\zeta(z+p)+\zeta(z-p)-2(r\eta_1+s\eta_2)=0.\end{equation}
By the Legendre relation $\tau\eta_1-\eta_2=2\pi i$, we easily obtain that \begin{equation}\label{rsab}-r\eta_1-s\eta_2=az+b\bar{z},\quad\text{where }a=\frac{\pi}{\operatorname{Im}\tau}-\eta_1,\quad b=-\frac{\pi}{\operatorname{Im}\tau}<0.\end{equation}
Inserting this into \eqref{a+1p}, we obtain
$$\zeta(z+p)+\zeta(z-p)+2(az+b\bar{z})=0,$$
or equivalently,
\begin{equation}\label{a+2p}
z=-\frac{1}{2b}\left(\overline{\zeta(z+p)}+\overline{\zeta(z-p)}+2\bar{a}\bar{z}\right)=:g(z).
\end{equation}
In conclusion, $z$ is a critical point of $G_p$ if and only if $g(z)=z$, i.e. $z$ is a fixed point of $g$.
Clearly $g(z)$ is an anti-meromorphic function in the plane, and a direct computation shows that
\begin{align}\label{g-w}
g(z+\omega)=g(z)+\omega,\quad\forall \omega\in \Lambda_{\tau}=\mathbb Z+\mathbb Z\tau.
\end{align}

\begin{Remark}
In \cite{BE}, Bergweiler and Eremenko studied the number of fixed points of the anti-meromorphic function\begin{equation}\label{hgz}\hat{g}(z):=-\frac{1}{b}\left(\overline{\zeta(z)}+\bar{a}\bar{z}\right)\end{equation}
to obtain the number of critical points of $G$. In particular, they proved that $\hat{g}$ has either $3$ or $5$ fixed points. In the following, we follow their argument to prove that $g$ has at most $10$ fixed points. We will point out the differences between $g$ and $\hat{g}$ during the following proof, which lead to the different upper bounds for the number of fixed points. 
\end{Remark}

To proceed, let us recall some notions as in \cite{BE}. A Riemann surface $S$ is called {\it hyperbolic} if its universal covering is the unit disk \cite{Ahlfors}. Let $S$ be a hyperbolic Riemann surface. For a holomorphic map $f: S\to S$, a point $z_0\in S$ is called {\it critical} if $f'(z_0)=0$; this definition does not depend on the local coordinates. A point $z_0\in S$ is called {\it fixed} if $f(z_0)=z_0$, and for such a point $f'(z_0)$ is called the {\it multiplier}. It is easy to verify that the multipier does not depend on the local coordinates. A fixed point $z_0$ is called {\it attracting}, {\it neutral} or {\it repelling} depending on whether $|f'(z_0)|$ is less than, equal to or greater than $1$, respectively. 

We say that a curve $\gamma: [0,1)\to S$ {\it escapes} if for every compact subset $K$ in $S$, there exists $t_0\in (0,1)$ such that $\gamma(t)\notin K$ for $t\in [t_0,1)$. A curve $\gamma: [0,1)\to S$ is called an {\it asymptotic curve} of $f: S\to S$ if $\gamma$ escapes and the limit $\lim_{t\to 1}f(\gamma(t))$ exists and is contained in $S$.
Recall the following Fatou's theorem.

\medskip
\noindent{\bf Fatou's theorem.} (see \cite[page 2914]{BE}) {\it Let $S$ be a hyperbolic Riemann surface and let $f: S\to S$ be a holomorphic map with an attracting fixed point $z_0\in S$. Then $f$ has a critical point or an asymptotic curve in $S$. Moreover, $f^n\to z_0$ locally uniformly in $S$}.
\medskip

In applications to holomorphic dynamics, $S$ is the immediate attraction basin of an attracting fixed point $z_0$; that is, the component of the Fatou set (see below for the definition) which contains $z_0$. Then it follows that the immediate attracting basin contains a critical point or an asymptotic curve.

Note that $g$ is anti-meromorphic. As in \cite{BE}, we apply the terminology attracting, neutral and repelling also to fixed points $z_0$ of anti-holomorphic maps $f$, considering $\overline{\partial}f(z_0)$ as the multiplier. Here $\partial=\partial_{z}$ and $\overline{\partial}=\partial_{\bar z}$. Then the kep step of proving Theorem \ref{main-thm-1} is to prove the following result.

\begin{Lemma}\label{lemma5}
The function $g$ has at most $4$ attracting fixed points (resp. at most $2$ attracting fixed points if $\wp(2p)=-2a=2\eta_1-\frac{2\pi}{\operatorname{Im}\tau}$), modulo $\Lambda_{\tau}$.
\end{Lemma}

Following \cite{BE}, we consider the set $P_0$ of poles of $g$. For $n\geq 1$ we define inductively $P_n=g^{-1}(P_{n-1})$ and $P_\infty=\cup_{n=0}^{+\infty} P_n$. Then all iterates of $g$ are defined on the set $\mathbb C\setminus P_\infty$. Let $J$ be the closure of $P_\infty$ and $F:=\mathbb{C}\setminus J$. As in \cite{BE}, we call $J$ the {\it Julia set} and $F$ the {\it Fatou set} of $g$. To apply Fatou's theorem, we need $F\neq \emptyset$. The non-empty of the Fatou set may be well known to experts in this field, but we can not find a proper reference for this fact. Since our goal is to prove Lemma \ref{lemma5}, we may assume that $g$ has at least one attracting fixed point. Then $F\neq\emptyset$ can be easily obtained as follows.

\begin{Lemma}
Suppose that $g$ has at least one attracting fixed point $z_0$, i.e. $g(z_0)=z_0$ and $|\overline{\partial}g(z_0)|<1$. Then $F\neq\emptyset$.
\end{Lemma}

\begin{proof} Under our assumption, there is small $\varepsilon_0>0$ such that $g(B_{\varepsilon_0}(z_0))\subset B_{\varepsilon_0}(z_0)$, where $B_{\varepsilon_0}(z_0):=\{z\in\mathbb{C}\;: \; |z-z_0|<\varepsilon\}$.
If there is $z_1\in B_{\varepsilon_0}(z_0)\cap P_\infty$, then there is $n\geq 1$ such that $g^n(z_1)=\infty$, a contradiction with $g^n(B_{\varepsilon_0}(z_0))\subset B_{\varepsilon_0}(z_0)$. Thus, $B_{\varepsilon_0}(z_0)\cap P_\infty=\emptyset$, which implies $B_{\varepsilon_0}(z_0)\cap J=\emptyset$, i.e. $B_{\varepsilon_0}(z_0)\subset F$ and so $F\neq \emptyset$. 
\end{proof}

The Fatou set is thus the maximal open subset of the plane such that $g(F)\subset F$. 
Furthermore, $F$ is $\Lambda_{\tau}$-invariant.
Let $\pi: \mathbb C\to E_{\tau}=\mathbb{C}/\Lambda_{\tau}$ be the projection map and denote $\widetilde F:=\pi(F)$. Then the map $g: F\to F$ descends to a map $\widetilde g: \widetilde F\to \widetilde F$ which satsifies $\widetilde g\circ \pi=\pi\circ g$. The Riemann surface $\widetilde F$ is a subset of the torus, and it is hyperbolic because its completement is infinite. 

Notice that $g: F\to F$ and $\widetilde g: \widetilde F\to \widetilde F$ are anti-holomorphic, and
\begin{align}\label{2-12}
\overline{\overline{\partial}g(z)}&=-\frac1{2b}\left(\zeta'(z+p)+\zeta'(z-p)+2a\right)\nonumber\\
&=\frac1{2b}\left(\wp(z+p)+\wp(z-p)-2a\right).
\end{align}
As in \cite{BE},
to obtain holomorphic dynamics instead of the anti-holomorphic one, we consider the second iterates $h=g^2$ and $\widetilde h=\widetilde g^2$. Then
\begin{equation}
\widetilde{h}\circ \pi=\pi\circ h,
\end{equation}
and images of fixed points of $g$ and $h$ in $F$ under $\pi$ are fixed points of $\widetilde g$ and $\widetilde h$, respectively. Furthermore, it follows from the chain role that
\begin{equation}\label{2-11}
h'=((\overline{\partial}g)\circ g)\cdot \partial \overline{g}=((\overline{\partial}g)\circ g)\cdot\overline{\overline{\partial}g}.
\end{equation}
Thus, for a fixed point $z_0$ of $g$, we have $h'(z_0)=|\overline{\partial}g(z_0)|^2$. We want to apply Fatou's theorem to $\widetilde h: \widetilde F\to \widetilde F$, so we need to consider the critical points and asymptotic curves of $\widetilde h$.

\begin{Lemma}
The map $\widetilde h: \widetilde F\to \widetilde F$ has no asymptotic curves.
\end{Lemma}

\begin{proof}
Thanks to \eqref{g-w}, the proof is the same as \cite[Lemma 1]{BE} where the same statement was proved for \eqref{hgz}, so we omit the details.
\end{proof}

As in \cite{BE}, we consider the equivalence relation on $\widetilde F$ defined by $z\sim z'$ if $\widetilde h^n(z)=\widetilde h^m(z')$ for some $n,m\in\mathbb N$. The equivalence classes are called the {\it grand orbits}. Notice that if $\widetilde h^n(z)$ converges for some $z\in\widetilde F$, as $n\to+\infty$, then for any $z'$ in the grand orbit of $z$, the sequence  $\widetilde h^n(z')$ converges to the same limit. We call this the limit of the grand orbit.

\begin{Lemma}\label{lemma2}
The set of critical points of $\widetilde h$ belongs to at most $8$ grand orbits (resp. at most $4$ grand orbits if $\wp(2p)=-2a$) under $\widetilde h$.
\end{Lemma}

\begin{proof}
By \eqref{2-12}, we know that $\overline{\partial}g(z)=0$ is equivalent to $\wp(z+p)+\wp(z-p)=2a$. Since the elliptic function $\wp(\cdot+p)+\wp(\cdot-p)$ has two double poles $\pm p$, it follows from the theory of elliptic functions that $\wp(z+p)+\wp(z-p)=2a$ has $4$  solutions modulo $\Lambda_{\tau}$, which defines two pairs of points $\pm c_1$ and $\pm c_2$ on $E_{\tau}$. 

If any of $\{\pm c_1, \pm c_2\}$ is in $\widetilde F$, then it is a zero of $\widetilde h'$. By \eqref{2-11}, we see that the other zeros of $\widetilde h'$ are $\widetilde g^{-1}(\pm c_j)$. Note that any point in $\widetilde g^{-1}(\pm c_j)$ belongs to the same grand orbit as $\widetilde g(\pm c_j)$. Therefore, the critical points of $\widetilde h$ are contained in at most $8$ grand orbits represented by $\pm c_j,   \widetilde g(\pm c_j)$, $j=1,2$.

Now we further assume $\wp(2p)=-2a$. Then by the additional formula of elliptic functions
$$\wp(u-v)=-\wp(u)-\wp(v)+\frac{(\wp'(u)+\wp'(v))^2}{4(\wp(u)-\wp(v))^2},$$
we see that $\wp(z+p)+\wp(z-p)=2a=-\wp(2p)$ implies $\wp'(z+p)+\wp'(z-p)=0,$
so $\wp(z+p)+\wp(z-p)=2a$ has only $2$ solutions $\pm c_1=\pm c_2$ of both multiplicities $2$ on $E_{\tau}$. Consequently, the critical points of $\widetilde h$ are contained in at most $4$ grand orbits represented by $\pm c_1,   \widetilde g(\pm c_1)$.
\end{proof}

Remark that this is the first difference from \cite{BE} where $\hat{g}(z)$ defined in \eqref{hgz} was studied. Since $\overline{\partial}\hat{g}(z)=0$ is equivalent to $\wp(z)=a$, which has $2$ solutions modulo $\Lambda_{\tau}$, it was proved in \cite{BE} that the set of critical points of the associated $\widetilde h$ of $\hat{g}$ belongs to at most $4$ grand orbits.

\begin{Lemma}\label{lemma3}
Let $z_0$ be a fixed point of $\widetilde g$. If $\widetilde h^n(z)\to z_0$ as $n\to\infty$ for some $z$, then $\widetilde h^n(\widetilde g(z))\to z_0$.
\end{Lemma}

\begin{proof}
This is a direct consequence of $\widetilde h^n\circ \widetilde g=\widetilde g\circ \widetilde h^n$.
\end{proof}

\begin{proof}[Proof of Lemma \ref{lemma5}]
Lemma \ref{lemma3} implies that if $\widetilde h^n(\pm c_j)\to z_0$ for some fixed point $z_0$ of $\widetilde g$, then $\widetilde h^n(\widetilde g(\pm c_j))\to z_0$. Clearly attracting fixed points of $\widetilde g$ are always attracting fixed points of $\widetilde h=\widetilde g^2$, so Fatou's theorem and Lemma \ref{lemma3} together imply that each attracting fixed point of $\widetilde{g}$ attracts at least two critical orbits of $\widetilde{h}$. Then by applying Fatou's theorem to $\widetilde{h}: S\to S$ (where $S$ is any connected component of $\widetilde{F}$ that contains an attracting fixed point of $\widetilde{g}$) and Lemma \ref{lemma2}, we deduce that $\widetilde g: \widetilde F\to \widetilde F$ has at most $4$ attracting fixed points (resp. at most $2$ attracting fixed points if $\wp(2p)=-2a$).

Since the $\pi$-image of an attracting fixed point of $g$ is an attracting fixed point of $\widetilde g$,
we obtain that $g$ has at most $4$  attracting fixed points (resp. at most $2$ attracting fixed points if $\wp(2p)=-2a$), modulo $\Lambda_{\tau}$.

Remark that this is different from \cite{BE}, where it was proved that $\hat{g}$ defined in \eqref{hgz} has at most $2$ attracting fixed points, modulo $\Lambda_{\tau}$.
\end{proof}

\begin{Remark}\label{rmk2-6}
For any $w\in \mathbb{C}$, we consider the function $g_{w}(z):=g(z)+w$. Then $g_{w}$ also satisfies \eqref{g-w}, i.e.
$$g_w(z+\omega)=g_w(z)+\omega,\quad\forall \omega\in \Lambda_{\tau}=\mathbb Z+\mathbb Z\tau,$$
and $\overline{\partial}g_w(z)\equiv\overline{\partial}g(z)$. Therefore, by applying the same arguments to $g_w$, we conclude that $g_w$ has at most $4$ attracting fixed points (resp. at most $2$ attracting fixed points if $\wp(2p)=-2a$), modulo $\Lambda_{\tau}$.
\end{Remark}

As in \cite{BE}, we will also use the following result from \cite{BE-2010}.

\begin{Lemma}\cite[Proposition 3]{BE-2010}\label{prop3}
Let $D\subset \mathbb C$ be a domain and let $f: D\to \mathbb C$ be a harmonic map. Suppose that there exists $m\in\mathbb{N}$ such that every $w\in\mathbb C$ has at most $m$ preimages. Then the set of points which have $m$ preimages is open.
\end{Lemma}

Now we are ready to prove Theorems \ref{thm-A00} and \ref{main-thm-1}.

\begin{proof}[Proof of Theorem \ref{main-thm-1}]
By \eqref{g-w}, we know that the map $\phi: E_{\tau}\to \overline{\mathbb C}=\mathbb C\cup\{\infty\}$, $z\mapsto \phi(z):=z-g(z)$ is well-defined. Notice that the set of fixed points of $g$ is precisely $\phi^{-1}(0)$, so our goal is to prove $N:=\# \phi^{-1}(0)\leq 10$ (resp. $N\leq 6$ if $\wp(2p)=-2a$). 

Let $J_{\phi}=1-|\overline{\partial}g|^2$ be the Jacobian determinant of $\phi$. Then $D^+=\{z\in E_{\tau} : J_\phi(z)>0\}$ is the set where $\phi$ preserves the orientation and $D^-=\{z\in E_{\tau} : J_\phi(z)<0\}$ is the set where $\phi$ reverses the orientation. Notice from \eqref{a+2p} that $\phi$ has two simple poles $\pm p$, then a point $w$ of large modulus has exactly two preimages in $E_{\tau}$ and the map is reversing the orientation at these two preimages. Therefore, the degree of $\phi$ equals to $-2$; see \cite[Section 5]{26} for the definition of the degree. (Remark that this is different from \cite{BE}, where $\hat{\phi}:=z-\hat{g}(z)$ has degree $-1$ because it has only a simple pole $0$.)

{\bf Step 1.}
Suppose first that $0$ is a regular value of $\phi$, i.e. $\phi^{-1}(0)\subset D^+\cup D^-$. Denote by $N^+=\# (\phi^{-1}(0)\cap D^+)$ and $N^-=\# (\phi^{-1}(0)\cap D^-)$. Then $N^+-N^-$ equals the degree of $\phi$, i.e. $N^+-N^-=-2$ and so
$$N=N^++N^-=2N^--2=2N^++2.$$
Since $J_{\phi}=1-|\overline{\partial}g|^2>0$ is equivalent to $|\overline{\partial}g|<1$, we see that the zeros of $\phi$ in $D^+$ are precisely the attracting fixed points of $g$, so Lemma \ref{lemma5} implies $N^+\leq 4$ (resp. $N^+\leq 2$ if $\wp(2p)=-2a$). Thus $N=2N^++2\leq 10$ (resp. $N\leq 6$ if $\wp(2p)=-2a$).

{\bf Step 2.}
For the general case that $0$ might be a singular value of $\phi$, we can modify the proof as follows. Let $w\in\mathbb{C}$ be any regular value of $\phi$, i.e. $\phi^{-1}(w)\subset D^+\cup D^-$. Note that $\phi(z)=w$ is equivalent to $z=g(z)+w=g_{w}(z)$, so  $\phi^{-1}(w)$ is precisely the set of fixed points of $g_{w}$. Thanks to Remark \ref{rmk2-6}, we can repeat the same argument as Step 1 to obtain that $\#\phi^{-1}(w)\leq 10$ (resp. $\#\phi^{-1}(w)\leq 6$ if $\wp(2p)=-2a$). Therefore, $\#\phi^{-1}(w)\leq 10$ (resp. $\#\phi^{-1}(w)\leq 6$ if $\wp(2p)=-2a$) for any regular value $w$ of $\phi$. 

Now in order to apply Lemma \ref{prop3}, we claim that there is $m\in \mathbb{N}_{\geq 10}$ such that $\#\phi^{-1}(w)\leq m$ for any $w\in\mathbb{C}$.

Recall that a point $w$ of large modulus has exactly two preimages in $E_{\tau}$, so we only need to consider $|w|\leq R$ for some large $R$. Note that $g_w(z)$ is anti-meromorphic, so $g_w^2(z)=g_w\circ g_w(z)$ is meromorphic. Define a meromphic function $\Phi_w: E_{\tau}\to \mathbb{C}\cup\{\infty\}$ by $\Phi_w(z):=z-g_w^2(z)$ for each $w$. Since $\phi(z)=w$ is equivalent to $z=g_w(z)$ and so $\Phi_w(z)=0$, we have $\phi^{-1}(w)\subset \Phi_w^{-1}(0)$. Since the zeros of $\Phi_w$ are isolated, so $\Phi_w^{-1}(0)$  is finite for every $w$. Furthermore, the argument principle implies that up to counting mulitiplicites, $\#\Phi_w^{-1}(0)$ is locally a constant for $w$. Therefore, the finite covering theorem implies the existence of $m\geq 10$ such that $\#\Phi_w^{-1}(0)\leq m$ for any $|w|\leq R$. 
This proves the claim.

Since Sard  theorem says that the set of singular values of $\phi$ is of  Lebesgue measure $0$, we conclude from Lemma \ref{prop3} that $\#\phi^{-1}(w)\leq 10$ (resp. $\#\phi^{-1}(w)\leq 6$ if $\wp(2p)=-2a$) for any $w\in\mathbb{C}$ and in particular, $\#\phi^{-1}(0)\leq 10$ (resp. $\#\phi^{-1}(0)\leq 6$ if $\wp(2p)=-2a$).

This proves that $G_p(z)$ has at most $10$ critical points (at most $6$ if $\wp(2p)=-2a$). Since $\frac{\omega_k}{2}$ are all critical points of $G_p$ for $k\in \{0,1,2,3\}$, and nontrivial critical points must appear in pairs, we conclude that the number of critical points of $G_p(z)$ belongs to $\{4,6,8,10\}$ (resp. $\{4,6\}$ if $\wp(2p)=-2a$).
The proof is complete.
\end{proof}

\begin{Remark}\label{section2-2} Recall $\mathcal{S}(\wp(p))$ defined in \eqref{sp}.
Given $t_1, t_2\in\mathbb{C}$ and suppose $(t_1+r, t_2+s)\in\mathcal{S}(\wp(p))$ with $(r,s)\in\mathbb R^2$. Then there exist $r_0, s_0\in\mathbb{R}$ such that $t_1+t_2\tau=r_0+s_0\tau$. Write $q=(t_1+r)+(t_2+s)\tau=(r_0+r)+(s_0+s)\tau$. Clearly $(t_1+r, t_2+s)\in\mathcal{S}(\wp(p))$ is equivalent to
$$\zeta(q+p)+\zeta(q-p)-2((t_1+r)\eta_1+(t_2+s)\eta_2)=0,$$
or equivalently,
$$\zeta(q+p)+\zeta(q-p)-2((r_0+r)\eta_1+(s_0+s)\eta_2)=2((t_1-r_0)\eta_1+(t_2-s_0)\eta_2).$$
This, together with \eqref{rsab}-\eqref{a+2p}, is equivalent to
\[
\phi(q)=q-g(q)=\frac{\overline{(t_1-r_0)\eta_1+(t_2-s_0)\eta_2}}{b}=:w.
\]
Therefore, $\#((\mathcal{S}(\wp(p))\cap ((t_1,t_2)+\mathbb R^2))/\mathbb{Z}^2)\leq \#\phi^{-1}(w)\leq 10$.
\end{Remark}

\begin{corollary}\label{coro2-8} Suppose $0$ is a regular value of $\phi$, i.e. $\phi^{-1}(0)\subset D^+\cup D^-$. Then
$\# (D^+\cap\{\frac{\omega_k}{2} \,:\, k=0,1,2,3\})\leq 2$. Furthermore,
\begin{itemize}
\item[(1)] If $\# (D^+\cap\{\frac{\omega_k}{2} \,:\, k=0,1,2,3\})=0$, then $\#\phi^{-1}(0)\in\{6,10\}$, namely $G_p(z)$ has exactly either $6$ or $10$ critical points, and both the two possibilities occur for different $(\tau, p)$'s.
\item[(2)] If $\# (D^+\cap\{\frac{\omega_k}{2} \,:\, k=0,1,2,3\})=1$, then $\#\phi^{-1}(0)\in\{4,8\}$, namely $G_p(z)$ has exactly either $4$ or $8$ critical points, and both the two possibilities occur for different $(\tau, p)$'s.
\item[(3)] If $\# (D^+\cap\{\frac{\omega_k}{2} \,:\, k=0,1,2,3\})=2$ (this case occurs for some $(\tau, p)$; see Section \ref{section5}), then $\#\phi^{-1}(0)=6$, namely $G_p(z)$ has exactly $6$ critical points.
\end{itemize}
\end{corollary}

\begin{proof}
Suppose $\frac{\omega_k}{2}\in D^+$ for some $k$. Then $\frac{\omega_k}{2}$ is an attracting fixed point of $g$ and hence attracts a critical orbit (say the orbit represented by $c_j$ for some $j\in \{1,2\}$ for example) by Fatou's theorem. Then $\widetilde{h}^n(c_j)\to \frac{\omega_k}{2}$. However, since $g$ is odd, we also have  $\widetilde{h}^n(-c_j)=-\widetilde{h}^n(c_j)\to -\frac{\omega_k}{2}=\frac{\omega_k}{2}$. Together with Lemma \ref{lemma3}, we see that $\frac{\omega_k}{2}$ actually attracts $4$ critical orbits represented by $\pm c_j,   \widetilde g(\pm c_j)$ respectively. From here and Lemma \ref{lemma2}, we obtain $\# (D^+\cap\{\frac{\omega_k}{2} \,:\, k=0,1,2,3\})\leq 2$. Furthermore, if $\# (D^+\cap\{\frac{\omega_k}{2} \,:\, k=0,1,2,3\})=2$, then these two attracting fixed points attract totally all the $8$ critical orbits, so there are no other attracting fixed points, which implies $N^+=2$ and so $N=2N^++2=6$. This proves (3).

(1) Suppose $\# (D^+\cap\{\frac{\omega_k}{2} \,:\, k=0,1,2,3\})=0$, i.e. $\{\frac{\omega_k}{2} \,:\, k=0,1,2,3\}\subset D^-$, so $N^-\geq 4$ and then $N=2N^--2\geq 6$. If $N^-=4$, we obtain $N=6$. If $N^-\geq 5$, i.e. there is a zero $q\notin\{\frac{\omega_k}{2} \,:\, k=0,1,2,3\}$ of $\phi$ such that $q\in D^-$, then $-q\neq q$ is also a zero of $\phi$ and $-q\in D^-$. Thus $N^-\geq 6$ and $N\geq 10$, which implies $N^-=6$ and $N=10$. We will see from Theorem \ref{thm-5c} and Theorem \ref{main-thm-2}-(1) that $N=6$ and $N=10$ both occur for different $(\tau, p)$'s.

(2) Suppose $\# (D^+\cap\{\frac{\omega_k}{2} \,:\, k=0,1,2,3\})=1$, then $\# (D^-\cap\{\frac{\omega_k}{2} \,:\, k=0,1,2,3\})=3$. The same argument as (1) implies $N^-\in\{3, 5\}$ and $N=2N^--2\in\{4, 8\}$. We will see from Theorem \ref{thm-B} and  Theorem \ref{main-thm-2}-(3-1) that $N=4$ and $N=8$ both occur for different $(\tau, p)$'s.
The proof is complete.
\end{proof}

\begin{proof}[Proof of Theorem \ref{thm-A00}]
Remark that in \cite{BE}, as mentioned above, since $\hat{\phi}=z-\hat{g}(z)$ has degree $-1$ and $\hat{g}$ has at most $2$ attracting fixed points, it follows that the number of fixed points of $\hat{g}$ equals $2N^++1\leq 5$. This is basically the proof in \cite{BE} that $G(z)$ has at most $5$ critical points.

Write $z=x+iy$ with $x,y\in\mathbb{R}$. By \eqref{G_z} and \eqref{rsab} we have
$$2\pi\frac{\partial G}{\partial x}=-\operatorname{Re}(\zeta(z)+az+b \bar{z}),$$
$$2\pi\frac{\partial G}{\partial y}=\operatorname{Im}(\zeta(z)+az+b \bar{z}),$$
and then
\begin{equation}\label{Gxx}2\pi\frac{\partial^2 G}{\partial x^2}=\operatorname{Re}(\wp(z)-a-b)=\operatorname{Re}(\wp(z)+\eta_1),\end{equation}
$$2\pi\frac{\partial^2 G}{\partial x\partial y}=2\pi\frac{\partial^2 G}{\partial y\partial x}=-\operatorname{Im}(\wp(z)-a)=-\operatorname{Im}(\wp(z)+\eta_1),$$
\begin{equation}\label{Gyy}2\pi\frac{\partial^2 G}{\partial y^2}=-\operatorname{Re}(\wp(z)-a)-b=-\operatorname{Re}(\wp(z)+\eta_1)+\frac{2\pi}{\operatorname{Im}\tau}.\end{equation}
Thus,
\begin{align}\label{detG}\det D^2 G(z)=&\det\begin{pmatrix}\frac{\partial^2 G}{\partial x^2}&\frac{\partial^2 G}{\partial x\partial y}\\\frac{\partial^2 G}{\partial y\partial x}&\frac{\partial^2 G}{\partial y^2}\end{pmatrix}
=\frac{\frac{2\pi}{\operatorname{Im}\tau}\operatorname{Re}(\wp(z)+\eta_1)-|\wp(z)+\eta_1|^2}{4\pi^2}\nonumber\\
=&\frac{1}{4(\operatorname{Im}\tau)^2}\Big(1-\Big|\frac{\wp(z)+\eta_1}{\pi}\operatorname{Im}\tau-1\Big|^2\Big).
\end{align}
On the other hand, since \eqref{hgz} and \eqref{rsab} imply
$$\overline{\overline{\partial}\hat{g}(z)}=\frac1{b}\left(\wp(z)-a\right)=-\Big(\frac{\wp(z)+\eta_1}{\pi}\operatorname{Im}\tau-1\Big),$$
we have
$$J_{\hat{\phi}}(z)=1-|\overline{\partial}\hat{g}(d)|^2=1-\Big|\frac{\wp(z)+\eta_1}{\pi}\operatorname{Im}\tau-1\Big|^2,$$
so
\begin{equation}\label{20-12}
\det D^2 G(z)=\frac{1}{4(\operatorname{Im}\tau)^2}J_{\hat{\phi}}(z).
\end{equation}

Suppose that $G(z)$ has $5$ critical points $\frac{\omega_k}{2}$, $k=1,2,3$, and $\pm q$.
For any critical point $d$ of $G(z)$, 
it follows from \eqref{20-12} that
$$\det D^2 G(d)>0\quad\text{if and only if }\quad J_{\hat{\phi}}(d)>0,$$
$$\det D^2 G(d)<0\quad\text{if and only if }\quad J_{\hat{\phi}}(d)<0,$$
$$\det D^2 G(d)=0\quad\text{if and only if }\quad J_{\hat{\phi}}(d)=0.$$
Then by applying Theorem \ref{thm-A0}, we have 
$$\frac{\omega_k}{2}\in \hat{D}^-:=\Big\{z\in E_{\tau} : J_{\hat{\phi}}(z)<0\Big\},\quad\forall k=1,2,3.$$
Denote also $\hat{D}^+:=\{z\in E_{\tau} : J_{\hat{\phi}}(z)>0\}$.

Assume by contradiction that $\det D^2 G(q)=0$, i.e. $J_{\hat{\phi}}(q)=0$, or equivalently, $q\in \partial \hat{D}^+=\partial\hat{D}^-$. Since $J_{\hat{\phi}}(z)=0$ if and only if
$$\wp(z)\in \partial \mathcal{B}_0=\Big\{z'\in\mathbb{C}\; :\; \Big|z'-\Big(\frac{\pi}{\operatorname{Im}\tau}-\eta_1\Big)\Big|=\frac{\pi}{\operatorname{Im}\tau}\Big\},$$
where $\mathcal{B}_0$ is defined in \eqref{B00}, we have $\partial\hat{D}^-=\wp^{-1}(\partial \mathcal{B}_0)$ and $\hat{D}^-=\wp^{-1}(\mathbb{C}\setminus\overline{\mathcal{B}_0})$. Denote $B_{\varepsilon}(q)=\{z\in\mathbb{C} : |z-q|<\varepsilon\}$. Then we can take $\varepsilon>0$ small such that $B_{\varepsilon}(q)\cap\partial\hat{D}^-$ is a connected analytic curve with $q$ being an interior point, and $B_{\varepsilon}(q)\cap\hat{D}^-$ is open and connected. Since the Jacobian determinant $J_{\hat{\phi}}(z)<0$ for any $z\in B_{\varepsilon}(q)\cap\hat{D}^-$, the inverse function theorem implies that $\hat{\phi}: B_{\varepsilon}(q)\cap\hat{D}^-\to\mathbb{C}$ is an open mapping, so $\hat{\phi}(B_{\varepsilon}(q)\cap\hat{D}^-)$ is open with $0=\hat{\phi}(q)\in \partial (\hat{\phi}(B_{\varepsilon}(q)\cap\hat{D}^-))$.
From here and Sard theorem, there is a regular value $w\in \hat{\phi}(B_{\varepsilon}(q)\cap\hat{D}^-)$ of $\hat{\phi}$ close to $0$ such that $\hat{\phi}^{-1}(w)$ contains a point in $B_{\varepsilon}(q)\cap\hat{D}^-$ that is close to $q$, and also contains a point in  $\hat{D}^-$ that is close to $\frac{\omega_k}{2}$ for each $k\in\{1,2,3\}$. This implies $\#(\hat{\phi}^{-1}(w)\cap \hat{D}^-)\geq 4$. Meanwhile, since $w$ is a regular value of $\hat{\phi}$, by \cite{BE} or by the same argument as the proof of Theorem \ref{main-thm-1}, we know that $\#\hat{\phi}^{-1}(w)\leq 5$, so $\#(\hat{\phi}^{-1}(w)\cap \hat{D}^+)\leq 1$. Together with the fact that the degree of $\hat{\phi}$ is $-1$, we obtain
$$-1=\#(\hat{\phi}^{-1}(w)\cap \hat{D}^+)-\#(\hat{\phi}^{-1}(w)\cap \hat{D}^-)\leq 1-4=-3,$$
a contradiction. Thus $\det D^2 G(-q)=\det D^2 G(q)\neq 0$. Then we must have $\pm q\in \hat{D}^+$, or equivalently, 
 $\det D^2 G(\pm q)> 0$. The proof is complete.
\end{proof}

\section{Criterion for degeneracy}
\label{section3}

In order to apply Corollary \ref{coro2-8}, we need the key assumption that $0$ is a regular value of $\phi$. We will see in Remark \ref{rmk3-11} that $0$ is a regular value of $\phi$ if and only if all critical points of $G_p(z)$ are non-degenerate. Therefore, it is important for us to study the non-degeneracy of critical points. In this section, we study the criterion for the degeneracy of trivial critical points, and prove Theorems \ref{main-thm-2} and \ref{main-thm-01}.

\subsection{Criterion for degeneracy of trivial critical points}

\begin{Lemma}\label{lemma3-1}
Let $p\in E_{\tau}\setminus E_{\tau}[2]$ and fix $k\in\{0,1,2,3\}$. Then $\frac{\omega_k}{2}$ is a degenerate critical point of $G_p(z)$ if and only if $\wp(p-\frac{\omega_k}{2})\in \partial\mathcal{B}_0$, where $\mathcal{B}_0$ is defined in \eqref{B00}.

In particular, there is $\varepsilon>0$ such that if $0<|p-\frac{\omega_k}{2}|<\varepsilon$, then the trivial critical point $\frac{\omega_k}{2}$ is non-degenerate.
\end{Lemma}

\begin{proof}
By \eqref{Gxx}-\eqref{Gyy}, \eqref{2-12} and $J_{\phi}=1-|\overline{\partial}g|^2$, a direct computation gives
\begin{align}\label{Gpdet}\det D^2 G_p(z)=&\frac{1}{4(\operatorname{Im}\tau)^2}\Big(1-\Big|\frac{\wp(z+p)+\wp(z-p)+2\eta_1}{2\pi}\operatorname{Im}\tau-1\Big|^2\Big)\nonumber\\=&\frac{1}{4(\operatorname{Im}\tau)^2}J_{\phi}(z).\end{align}
Thus, for the trivial critical point $\frac{\omega_k}{2}$, since $\wp(p+\frac{\omega_k}{2})=\wp(p-\frac{\omega_k}{2})=\wp(\frac{\omega_k}{2}-p)$, we have
{\allowdisplaybreaks
\begin{align}\label{eq3-2}\det D^2 G_p\Big(\frac{\omega_k}{2}\Big)&=\frac{1}{4(\operatorname{Im}\tau)^2}J_{\phi}\Big(\frac{\omega_k}{2}\Big)\nonumber\\
&=\frac{1}{4(\operatorname{Im}\tau)^2}\Big(1-\Big|\frac{\wp(p-\frac{\omega_k}{2})+\eta_1}{\pi}\operatorname{Im}\tau-1\Big|^2\Big)\nonumber\\
&=\frac{1}{4\pi^2}\bigg(\frac{\pi^2}{(\operatorname{Im}\tau)^2}-\Big|\wp\Big(p-\frac{\omega_k}{2}\Big)-\Big(\frac{\pi}{\operatorname{Im}\tau}-\eta_1\Big)\Big|^2\bigg).\end{align}
}%
Recalling the open disk $\mathcal{B}_0$ defined in \eqref{B00}, 
it follows that the trivial critical point $\frac{\omega_k}{2}$ is degenerate if and only if
\begin{align}\label{S0}\wp\Big(p-\frac{\omega_k}{2}\Big)\in \partial\mathcal{B}_0
=\Big\{z\in\mathbb{C}\; :\; \Big|z-\Big(\frac{\pi}{\operatorname{Im}\tau}-\eta_1\Big)\Big|=\frac{\pi}{\operatorname{Im}\tau}\Big\}.\end{align}
 In particular, since $\wp(p-\frac{\omega_k}{2})\to\infty$ as $p\to \frac{\omega_k}{2}$, there is $\varepsilon>0$ such that if $0<|p-\frac{\omega_k}{2}|<\varepsilon$, the trivial critical point $\frac{\omega_k}{2}$ is non-degenerate. 
 This completes the proof.
\end{proof}

\begin{Remark}\label{rmk3-11}
For any critical point $q$ of $G_p(z)$, we see from \eqref{Gpdet} that
$$\det D^2 G_p(q)>0\;\Leftrightarrow\; J_{\phi}(q)>0\;\Leftrightarrow\;\left|\frac{\wp(q+p)+\wp(q-p)+2\eta_1}{2\pi}\operatorname{Im}\tau-1\right|<1,$$
$$\det D^2 G_p(q)<0\;\Leftrightarrow\; J_{\phi}(q)<0\;\Leftrightarrow\;\left|\frac{\wp(q+p)+\wp(q-p)+2\eta_1}{2\pi}\operatorname{Im}\tau-1\right|>1,$$
$$\det D^2 G_p(q)=0\;\Leftrightarrow\; J_{\phi}(q)=0\;\Leftrightarrow\;\left|\frac{\wp(q+p)+\wp(q-p)+2\eta_1}{2\pi}\operatorname{Im}\tau-1\right|=1.$$
Therefore, $0$ is a regular value of $\phi$ if and only if all critical points of $G_p(z)$ are non-degenerate, and the degree counting formula \eqref{deg-count} holds.
In particular, by \eqref{eq3-2}, we conclude that for $q=\frac{\omega_k}{2}$,
\begin{equation}\label{eq330}\det D^2 G_p\Big(\frac{\omega_k}{2}\Big)>0\;\Leftrightarrow\; J_{\phi}\Big(\frac{\omega_k}{2}\Big)>0\;\Leftrightarrow\;\wp\Big(p-\frac{\omega_k}{2}\Big)\in \mathcal{B}_0,\end{equation}
$$\det D^2 G_p\Big(\frac{\omega_k}{2}\Big)<0\;\Leftrightarrow\; J_{\phi}\Big(\frac{\omega_k}{2}\Big)<0\;\Leftrightarrow\;\wp\Big(p-\frac{\omega_k}{2}\Big)\in \mathbb{C}\setminus\overline{\mathcal{B}_0},$$
$$\det D^2 G_p\Big(\frac{\omega_k}{2}\Big)=0\;\Leftrightarrow\; J_{\phi}\Big(\frac{\omega_k}{2}\Big)=0\;\Leftrightarrow\;\wp\Big(p-\frac{\omega_k}{2}\Big)\in \partial\mathcal{B}_0.$$
\end{Remark}

\begin{Lemma}\label{coro3-1}
For $p\in E_{\tau}\setminus E_{\tau}[2]$ and $k\in \{1,2,3\}$, $\frac{\omega_k}{2}$ is a degenerate critical point of $G_p(z)$ if and only if 
\[\wp(p)\in\partial\mathcal{B}_k=\begin{cases}\bigg\{z\in\mathbb{C}\; :\; \bigg|z-e_k-\frac{\overline{\alpha_k}}{|\alpha_k|^2-\beta_k^2}\bigg|=\frac{\beta_k}{\left||\alpha_k|^2-\beta_k^2\right|}\bigg\}\;\text{if}\;|\alpha_k|\neq \beta_k,\\
\Big\{z\in\mathbb{C}\; :\; \operatorname{Re}(\alpha_k (z-e_k))=\frac12\Big\}\;\text{if}\;|\alpha_k|=\beta_k.\end{cases}\]
Here $\alpha_k, \beta_k$ and $\mathcal{B}_k$ are defined in \eqref{alphak0}-\eqref{alphak1}.
\end{Lemma}

\begin{proof}
Let $k\in \{1,2,3\}$. By applying the additional formula of elliptic functions,
we have
$$\wp\Big(p-\frac{\omega_k}{2}\Big)=e_k+\frac{\wp''(\frac{\omega_k}{2})}{2(\wp(p)-e_k)}=e_k+\frac{3e_k^2-\frac{g_2}{4}}{\wp(p)-e_k}.$$
Note $3e_k^2-\frac{g_2}{4}\neq 0$.
Then the trivial critical point $\frac{\omega_k}{2}$ is degenerate if and only if $\wp(p-\frac{\omega_k}{2})\in \partial\mathcal{B}_0$, if and only if
$$\Big|e_k+\frac{3e_k^2-\frac{g_2}{4}}{\wp(p)-e_k}-\Big(\frac{\pi}{\operatorname{Im}\tau}-\eta_1\Big)\Big|=\frac{\pi}{\operatorname{Im}\tau},$$
which, together with \eqref{alphak0}, is equivalent to
\begin{equation}\label{betakk}
\Big|\frac{1}{\wp(p)-e_k}-\alpha_k\Big|=\beta_k.
\end{equation}
There are two cases.

{\bf Case 1.} $|\alpha_k|\neq \beta_k$. Notice from \eqref{alphak0} that this is equivalent to
\begin{equation}
\left|\frac{\eta_1+e_k}{\pi}\operatorname{Im}\tau-1\right|\neq 1,
\end{equation}
which, together with \eqref{detG}, is equivalent to that
 $\frac{\omega_k}{2}$ is a non-degenerate critical point of $G(z)$, and $|\alpha_k|> \beta_k$
if and only if $\det D^2G(\frac{\omega_k}{2})<0$.

Then a direct computation shows that \eqref{betakk} is equivalent to
$$\bigg|\wp(p)-e_k-\frac{\overline{\alpha_k}}{|\alpha_k|^2-\beta_k^2}\bigg|=\frac{\beta_k}{\left||\alpha_k|^2-\beta_k^2\right|},$$
so $\wp(p-\frac{\omega_k}{2})\in \partial\mathcal{B}_0$ is equivalent to
\begin{equation}\label{Sk-11}
\wp(p)\in\partial\mathcal{B}_k=\bigg\{z\in\mathbb{C}\; :\; \bigg|z-e_k-\frac{\overline{\alpha_k}}{|\alpha_k|^2-\beta_k^2}\bigg|=\frac{\beta_k}{\left||\alpha_k|^2-\beta_k^2\right|}\bigg\}.
\end{equation}

{\bf Case 1-1.} $|\alpha_k|>\beta_k$, i.e. $\det D^2G(\frac{\omega_k}{2})<0$. Then $\wp(p-\frac{\omega_k}{2})\in \mathcal{B}_0$ is equivalent to $|\frac{1}{\wp(p)-e_k}-\alpha_k|<\beta_k$, which is equivalent to 
$$\bigg|\wp(p)-e_k-\frac{\overline{\alpha_k}}{|\alpha_k|^2-\beta_k^2}\bigg|<\frac{\beta_k}{\left||\alpha_k|^2-\beta_k^2\right|},$$
i.e. $\wp(p)\in \mathcal{B}_k$. Together with Remark \ref{rmk3-11}, we obtain that 
\begin{equation}\label{eq331}
\det D^2 G_p\Big(\frac{\omega_k}{2}\Big)>0\;\Leftrightarrow\; J_{\phi}\Big(\frac{\omega_k}{2}\Big)>0\;\Leftrightarrow\;\wp(p)\in \mathcal{B}_k,
\end{equation}
$$\det D^2 G_p\Big(\frac{\omega_k}{2}\Big)<0\;\Leftrightarrow\; J_{\phi}\Big(\frac{\omega_k}{2}\Big)<0\;\Leftrightarrow\;\wp(p)\in \mathbb C\setminus \overline{\mathcal{B}_k},$$
\begin{equation}\label{eq331-0}\det D^2 G_p\Big(\frac{\omega_k}{2}\Big)=0\;\Leftrightarrow\; J_{\phi}\Big(\frac{\omega_k}{2}\Big)=0\;\Leftrightarrow\;\wp(p)\in \partial\mathcal{B}_k.\end{equation}

{\bf Case 1-2.} $|\alpha_k|<\beta_k$, i.e. $\det D^2G(\frac{\omega_k}{2})>0$. Then $\wp(p-\frac{\omega_k}{2})\in \mathcal{B}_0$ is equivalent to $|\frac{1}{\wp(p)-e_k}-\alpha_k|<\beta_k$, which is equivalent to 
$$\bigg|\wp(p)-e_k-\frac{\overline{\alpha_k}}{|\alpha_k|^2-\beta_k^2}\bigg|>\frac{\beta_k}{\left||\alpha_k|^2-\beta_k^2\right|},$$
i.e. $\wp(p)\in \mathbb C\setminus \overline{\mathcal{B}_k}$. Thus 
\begin{equation}\label{eq332}
\det D^2 G_p\Big(\frac{\omega_k}{2}\Big)>0\;\Leftrightarrow\; J_{\phi}\Big(\frac{\omega_k}{2}\Big)>0\;\Leftrightarrow\;\wp(p)\in \mathbb C\setminus \overline{\mathcal{B}_k},
\end{equation}
$$\det D^2 G_p\Big(\frac{\omega_k}{2}\Big)<0\;\Leftrightarrow\; J_{\phi}\Big(\frac{\omega_k}{2}\Big)<0,\;\Leftrightarrow\;\wp(p)\in  {\mathcal{B}_k},$$
$$\det D^2 G_p\Big(\frac{\omega_k}{2}\Big)=0\;\Leftrightarrow\; J_{\phi}\Big(\frac{\omega_k}{2}\Big)=0\;\Leftrightarrow\;\wp(p)\in \partial\mathcal{B}_k.$$

{\bf Case 2.} $|\alpha_k|=\beta_k$, which is equivalent to that
 $\frac{\omega_k}{2}$ is a degenerate critical point of $G(z)$.
 Again a direct computation shows that \eqref{betakk} is equivalent to $\operatorname{Re}(\alpha_k (\wp(p)-e_k))=\frac12$, so $\wp(p-\frac{\omega_k}{2})\in \partial\mathcal{B}_0$ is equivalent to
 \begin{equation}\label{Sk-12}
\wp(p)\in\partial\mathcal{B}_k=\Big\{z\in\mathbb{C}\; :\; \operatorname{Re}(\alpha_k (z-e_k))=\frac12\Big\},
\end{equation} 
and similarly, $\wp(p-\frac{\omega_k}{2})\in \mathcal{B}_0$ is equivalent to
$$\wp(p)\in\mathcal{B}_k=\Big\{z\in\mathbb{C}\; :\; \operatorname{Re}(\alpha_k (z-e_k))>\frac12\Big\}.$$
Together with Remark \ref{rmk3-11}, we see that \eqref{eq331}-\eqref{eq331-0} hold.

The proof is complete.
\end{proof}

\begin{proof}[Proof of Theorem \ref{main-thm-01}]
Theorem \ref{main-thm-01} follows directly from Lemmas \ref{lemma3-1} and \ref{coro3-1}. 
Consequently, define
\begin{align}\label{O0tau}\mathcal{T}_{\tau}:=&\Big\{p\in E_{\tau}\setminus E_{\tau}[2]\;:\;\text{$G_p(z)$ has degenerate trivial critical points}\Big\}\nonumber\\
=&\bigcup_{k=0}^3\wp^{-1}(\partial\mathcal{B}_k)\setminus E_{\tau}[2].\end{align}
Since $\wp: E_{\tau}\to \mathbb{C}\cup\{\infty\}$ is a double cover with branch points at $\frac{\omega_k}{2}$'s, so $\wp^{-1}(\partial \mathcal{B}_k)$ consists of at most $2$ analytic curves, and then $\mathcal{T}_{\tau}\cup E_{\tau}[2]$ consists of at most $8$ analytic curves  (and so of Lebesgue measure zero).
\end{proof}

\begin{Remark}
Note that $0\notin\wp^{-1}(\partial \mathcal{B}_0)$ and $\wp^{-1}(\partial \mathcal{B}_0)=-\wp^{-1}(\partial \mathcal{B}_0)$ consists of either one or two analytic curves that are symmetric with respect to $0$. Then
$$\wp^{-1}(\partial \mathcal{B}_0)+\frac{\omega_k}{2}\neq \wp^{-1}(\partial \mathcal{B}_0)+\frac{\omega_j}{2},\quad\text{for }\;j\neq k.$$
 Since Lemmas \ref{lemma3-1} and \ref{coro3-1} together imply $\wp^{-1}(\partial \mathcal{B}_k)=\wp^{-1}(\partial \mathcal{B}_0)+\frac{\omega_k}{2}$, we obtain \begin{align}\label{Bjk}\partial{B}_j\neq \partial \mathcal B_k,\quad \text{for }\;j\neq k.\end{align}
\end{Remark}

\subsection{Examples that evey number in $\{4,6,8,10\}$ occurs}
Now we want to give examples to show that for the number of critical points of $G_p(z)$, every number in $\{4,6,8,10\}$ really occurs. The following result from \cite{CKL-2025} proves that the number $4$ occurs.

\begin{theorem}\cite[Theorem 1.5]{CKL-2025}\label{thm-B}
Fix any $\tau$ (such as $\tau\in i\mathbb{R}_{>0}$) such that $G(z;\tau)$ has exactly $3$ critical points that are all non-degenerate. Then there is $\varepsilon>0$ small such that if $0<|p-\frac{\omega_k}{2}|<\varepsilon$ for some $k\in \{0,1,2,3\}$, then $G_p(z)$ has no nontrivial critical points, or equivalently, has exactly $4$ critical points $\frac{\omega_k}{2}$, $k=0,1,2,3$, that are all non-degenerate.
\end{theorem}

\begin{proof}
For the case $\tau\in i\mathbb{R}_{>0}$, this theorem follows directly from \cite[Theorem 1.5]{CKL-2025} and Lemma \ref{lemma3-1}. Remak that in \cite[Theorem 1.5]{CKL-2025},
the condition $\tau\in i\mathbb{R}_{>0}$ was only used to imply that $G(z;\tau)$ has exactly $3$ critical points that are all non-degenerate. Thus, the proof of \cite[Theorem 1.5]{CKL-2025} also works for any $\tau$ such that $G(z;\tau)$ has exactly $3$ critical points that are all non-degenerate.
\end{proof}

Here we can prove the following result as a counterpart of Theorem \ref{thm-B} for those $\tau$ such that $G(z)=G(z;\tau)$ has $5$ critical points.

\begin{theorem}\label{thm-5c}
Fix any $\tau$ such that $G(z)=G(z;\tau)$ has $5$ critical points. Then there exists $\varepsilon>0$ small such that if $0<|p-\frac{\omega_k}{2}|<\varepsilon$ for some $k\in \{0,1,2,3\}$, $G_p(z;\tau)$ has exactly $6$ critical points, which are all non-degenerate. Furthermore, the unique pair of nontrivial critical points are non-degenerate minimal points of $G_p(z)$.
\end{theorem}

\begin{proof}
Fix any $\tau$ such that $G(z)=G(z;\tau)$ has $5$ critical points $\frac{\omega_k}{2}$, $k=1,2,3$, and $\pm q\notin E_{\tau}[2]$. It follows from Theorems \ref{thm-A0} and \ref{thm-A00} that these critical points are all non-degenerate and 
$$\det D^2G(\pm q)>0,\quad\text{and}\quad \det D^2G\Big(\frac{\omega_k}{2}\Big)<0\;\text{for}\;k=1,2,3.$$
Then by the implicit function theorem, there is $\varepsilon>0$ small such that for any $0<|p|<\varepsilon$, $G_p(z)$ has a unique critical point (which is non-degenerate) in a small neighborhood of each critical point of $G(z)$ (clearly the unique critical point of $G_p(z)$ in the small neighborhood of $\frac{\omega_k}{2}$ is also $\frac{\omega_k}{2}$ for $k=1,2,3$). Denote by $\pm q_p$ to be the unique critical point  of $G_p(z)$ in the small neighborhood of $\pm q$. Then
$$\det D^2G_p(\pm q_p)>0,\quad\text{and}\quad \det D^2G_p\Big(\frac{\omega_k}{2}\Big)<0\;\text{for}\;k=1,2,3.$$
Together with the trivial critical point $0$ which satisfies $\det D^2G_p(0)<0$ by using \eqref{eq330} and $\wp(p)\notin \overline{\mathcal B_0}$, we see that for any $0<|p|<\varepsilon$, $G_p(z)$ already has $6$ non-degenerate critical points $\frac{\omega_k}{2}$, $k=0,1,2,3$, and $\pm q_p\notin E_{\tau}[2]$. 

Assume by contradiction that there is $p_n\to 0$ such that $G_{p_n}(z)$ has more than $6$ critical points, i.e. has another critical point $z_{p_n}\notin E_{\tau}[2]\cup\{\pm q_{p_n}\}$. Then $z_{p_n}$ is a nontrivial critical point. Up to a subsequence, we may assume $z_{p_n}\to z_0\in E_{\tau}$. 
Suppose $z_0=0$ in $E_{\tau}$, by replacing $z_{p_n}$ with some element from $z_{p_n}+\mathbb{Z}+\mathbb{Z}\tau$, we may assume $z_{p_n}=r_n+s_n\tau\to 0$ with $\mathbb R\ni r_n, s_n\to 0$. Then 
 it follows from the Laurent series of $\zeta(z)$ and $\wp(z)$ that
\[
\zeta(z_{p_n})=\frac{1}{z_{p_n}}+O(|z_{p_n}|^{3}),\quad
\wp(z_{p_n})=\frac{1}{z_{p_n}^{2}}+O(|z_{p_n}|^{2}),
\]%
\[
\wp^{\prime}(z_{p_n})=\frac{-2}{z_{p_n}^{3}}+O(|z_{p_n}|).
\]
Inserting these into \eqref{513-1}, we obtain
\begin{align*}
\wp(p_n)=\wp(z_{p_n})+\frac{\wp^{\prime
}(z_{p_n})}{2(\zeta(z_{p_n})-r_n\eta_1-s_n\eta_2) }=-\frac{r_n\eta_1+s_n\eta_2}{z_{p_n}}+O(|z_{p_n}|^{2}).
\end{align*}
Note from $\operatorname{Im}\tau>0$ and $\mathbb{R}^2\ni (r_n, s_n)\to (0,0)$ that $\frac{r_n\eta_1+s_n\eta_2}{z_{p_n}}=\frac{r_n\eta_1+s_n\eta_2}{r_n+s_n\tau}$ is uniformly bounded, we obtain a contradiction with $\wp(p_n)\to \infty$ as $p_n\to 0$. 
This proves $z_0\neq 0$ in $E_{\tau}$, so $z_0$ is a critical point of $G(z)$, i.e. $z_0=\frac{\omega_k}{2}$ for some $k=1,2,3$ or $z_0=\pm q$. But this is a contradiction with the fact that the critical point of $G_{p_n}(z)$ in the small neighborhood of $z_0$ is unique. Therefore, by taking $\varepsilon>0$ smaller if necessary, we conclude that for any $0<|p|<\varepsilon$, $G_p(z)$ exactly $6$ critical points that are all non-degenerate, with $\det D^2G_p(\frac{\omega_k}{2})<0$ for $k\in\{0,1,2,3\}$ and $\det D^2G_p(\pm q_p)>0$, namely $\pm q_p$ are minimal points of $G_p(z)$. This is the case in Corollary \ref{coro2-8}-(1) with $6$ critical points.

Finally, since $G_{p}(z)=G_{p-\frac{\omega_k}{2}}(z+\frac{\omega_k}{2})$, the same assertion also holds for any $0<|p-\frac{\omega_k}{2}|<\varepsilon$.
The proof is complete.
\end{proof}

Now we turn to the proof of Theorem \ref{main-thm-2} (1)-(3) (the proof of Theorem \ref{main-thm-2} (3-1)-(3-2) will be postponed in Section \ref{section5}). Recall the following result from \cite{CKLW}.

\begin{theorem}\cite[Theorem 1.2]{CKLW}\label{thm-CKLW}
Define
$$\Omega_5:=\big\{\tau\in \mathbb{H}\;:\; G(z;\tau)\; \text{has $5$ critical points}\big\},$$
$$\Omega_3:=\big\{\tau\in \mathbb{H}\;:\; G(z;\tau)\; \text{has $3$ critical points}\big\}.$$
Then $\Omega_5$ is open and $\Omega_3^{\circ}\neq \emptyset$, where $\Omega_3^{\circ}$ denotes the subset of interior points of $\Omega_3$. Furthermore,
\begin{itemize}
\item[(1)] $\mathcal{C}:=\partial \Omega_5\cap \mathbb H=\partial \Omega_3\cap \mathbb H$ consists of analytic curves in $\mathbb{H}$, and $\mathcal{C}$ consists of those $\tau$ so that some half-period $\frac{\omega_k}{2}$ is degenerate critical point of $G(z;\tau)$.
\item[(2)]
Fix $k\in\{1,2,3\}$. Then there exists $\tau\in\mathcal{C}$ such that $\frac{\omega_k}{2}$ is a degenerate critical point of $G(z;\tau)$. In this case,  $\frac{\omega_l}{2}$ are non-degenerate saddle points of $G(z;\tau)$ for each $l\in\{1,2,3\}\setminus\{k\}$.
\end{itemize}
\end{theorem}

\begin{proof}[Proof of Theorem \ref{main-thm-2} (1)-(3)] (1) First, we consider $p=\frac{\omega_k}{4}$ for $k\in\{1,2,3\}$. Let us take $k=2$, i.e. $p=\frac{\omega_2}{4}=\frac{\tau}{4}$ for example. Fix any $\tau$ such that the Green function $G(z;\frac{\tau}{2})$ (i.e. the Green function on the torus $E_{\frac{\tau}{2}}=\mathbb{C}/(\mathbb Z+\mathbb Z\frac{\tau}{2})$) has $5$ critical points
\begin{equation}\label{500}\frac{1}{2}, \quad\frac{\tau}{4}, \quad\frac12+\frac{\tau}{4},\quad \pm q\notin E_{\frac{\tau}{2}}[2].\end{equation}
It follows from Theorems \ref{thm-A0} and \ref{thm-A00} that these critical points are all non-degenerate:
\begin{equation}\label{detd2}\det D^2G\Big(\pm q;\frac{\tau}{2}\Big)>0,\;\text{and}\; \det D^2G\Big(z;\frac{\tau}{2}\Big)<0\;\text{for}\;z\in\Big\{\frac{1}{2}, \frac{\tau}{4}, \frac12+\frac{\tau}{4}\Big\}.\end{equation}

Since $G(z)=G(z;\tau)$ is doubly periodic with periods $1$ and $\tau$, we see that $G_{\frac{\tau}{4}}(z)=\frac12(G(z+\frac{\tau}{4})+G(z-\frac{\tau}{4}))$ satisfies
\[G_{\frac{\tau}{4}}\Big(z+\frac{\tau}{2}\Big)=\frac{G(z+\frac{3\tau}{4})+G(z+\frac{\tau}{4})}{2}=\frac{G(z-\frac{\tau}{4})+G(z+\frac{\tau}{4})}{2}=G_{\frac{\tau}{4}}(z),\]
namely $G_{\frac{\tau}{4}}(z)$ is doubly periodic with periods $1$ and $\frac{\tau}{2}$, so $G_{\frac{\tau}{4}}(z)$ is well-defined on $E_{\frac{\tau}{2}}$. Furthermore, since $\frac{\tau}{4}=-\frac{\tau}{4}$ in $E_{\frac{\tau}{2}}$, it follows that
$$-\Delta (2G_{\frac{\tau}{4}}(z))=\delta_{\frac{\tau}{4}}-\frac{1}{\left \vert E_{\frac{\tau}{2}}\right \vert }\text{
\ on }E_{\frac{\tau}{2}},$$
i.e. $2G_{\frac{\tau}{4}}(z)$ is the Green function of $E_{\frac{\tau}{2}}$ with singularity at $\frac{\tau}{4}$. By the uniqueness of the Green function up to adding a constant, it follows that
 there is a constant $C$ such that 
$$2G_{\frac{\tau}{4}}(z)=G\Big(z-\frac{\tau}{4};\frac{\tau}{2}\Big)+C.$$ 
Then \eqref{500} implies that $G_{\frac{\tau}{4}}(z)$ has $5$ non-degenerate critical points on $E_{\frac{\tau}{2}}$:
$$\frac{1}{2}+\frac{\tau}{4}, \quad0, \quad\frac12,\quad q+\frac{\tau}{4}\notin E_{\frac{\tau}{2}}[2], \quad -q+\frac{\tau}{4}\notin E_{\frac{\tau}{2}}[2].$$
Since $z=z+\frac{\tau}{2}$ in $E_{\frac{\tau}{2}}$ but $z\neq z+\frac{\tau}{2}$ in $E_{\tau}$, we finally conclude that 
$G_{\frac{\tau}{4}}(z)$ has $10$ non-degenerate critical points on $E_{{\tau}}$:
$$ 0,\quad \frac12, \quad\frac{\tau}{2},\quad\frac{1+\tau}{2},\quad \pm\Big(\frac{1}{2}+\frac{\tau}{4}\Big),\quad \pm \Big(q+\frac{\tau}{4}\Big),\quad \pm \Big(q-\frac{\tau}{4}\Big).
$$
Furthermore, it follows from \eqref{detd2} that
$$\det D^2G_{\frac{\tau}{4}}(z)<0\quad\text{for}\quad z\in\Big\{0, \frac12, \frac{\tau}{2},\frac{1+\tau}{2}, \pm\Big(\frac{1}{2}+\frac{\tau}{4}\Big)\Big\},$$
$$\det D^2G_{\frac{\tau}{4}}(z)>0\quad\text{for}\quad z\in\Big\{\pm \Big(q+\frac{\tau}{4}\Big), \pm \Big(q-\frac{\tau}{4}\Big)\Big\}.$$
This is the case in Corollary \ref{coro2-8}-(1) with $10$ critical points. 

Consequently, by the implicit function theorem, there is $\varepsilon>0$ small such that for any $0<|p-\frac{\tau}{4}|<\varepsilon$, $G_p(z)$ has a unique critical point (which is also non-degenerate) in a small neighborhood of each critical point of $G_{\frac{\tau}{4}}(z)$. This implies that $G_p(z)$ has at least $10$ non-degenerate critical points. Together with Theorem \ref{main-thm-1}, we conclude that  $G_p(z)$ has exactly $10$ critical points that are all non-degenerate for any $|p-\frac{\tau}{4}|<\varepsilon$. This proves the assertion (1) for $k=2$. 

For $k=1$ (resp. $k=3$), it suffices to fix any $\tau$ such that the Green function $G(z;\frac12,\tau)$ on the torus $\mathbb{C}/(\mathbb Z\frac12+\mathbb Z{\tau})$ (resp. the Green function $G(z;\frac{1+\tau}{2})$ on the torus $E_{\frac{1+\tau}{2}}=\mathbb{C}/(\mathbb Z+\mathbb Z\frac{1+\tau}{2}$)) has $5$ critical points. Then the rest proof is the similar and we omit the details.

(2)-(3) Again let us take $k=2$, i.e. $p=\frac{\omega_2}{4}=\frac{\tau}{4}$ for example (again the cases $k\in\{1,3\}$ are similar and we omit the details). Fix any $\tau$ such that the Green function $G(z;\frac{\tau}{2})$ on the torus $E_{\frac{\tau}{2}}=\mathbb{C}/(\mathbb Z+\mathbb Z\frac{\tau}{2})$ has exactly $3$ critical points $\frac12, \frac{\tau}{4}, \frac{1}{2}+\frac{\tau}{4}$. Then the same proof as (1) implies that $G_{\frac{\tau}{4}}(z)$ has exactly $6$ critical points
$$ 0,\quad \frac12, \quad\frac{\tau}{2},\quad\frac{1+\tau}{2},\quad \pm\Big(\frac{1}{2}+\frac{\tau}{4}\Big),
$$
with $\{0,\frac{\tau}{2}\}$ coming from the critical point $\frac{\tau}{4}$ of $G(z;\frac{\tau}{2})$, $\{\frac12, \frac{1+\tau}{2}\}$ coming from $\frac{1}{2}+\frac{\tau}{4}$ and $\{\pm(\frac{1}{2}+\frac{\tau}{4})\}$ coming from $\frac12$.

{\bf Case 1.} By Theorem \ref{thm-CKLW}, we can further take $\tau$ such that $\frac12$ is a degenerate critical point of $G(z;\frac{\tau}{2})$ and $\det D^2G(z;\frac{\tau}{2})<0$ for $z\in\{\frac{\tau}{4}, \frac{1}{2}+\frac{\tau}{4}\}$. Then
$$\det D^2G_{\frac{\tau}{4}}(z)<0\quad\text{for}\quad z\in\Big\{0, \frac12, \frac{\tau}{2},\frac{1+\tau}{2}\Big\},$$
and the nontrivial critical points $\pm(\frac{1}{2}+\frac{\tau}{4})$ are degenerate minimal points of $G_{\frac{\tau}{4}}(z)$. This proves the assertion (2).

{\bf Case 2.} By Theorem \ref{thm-CKLW}, we can further take $\tau$ such that $\frac{\tau}{4}$ is a degenerate critical point of $G(z;\frac{\tau}{2})$ and $\det D^2G(z;\frac{\tau}{2})<0$ for $z\in\{\frac{1}{2}, \frac{1}{2}+\frac{\tau}{4}\}$. Then $\det D^2G_{\frac{\tau}{4}}(z)=0$ for $z\in \{0,\frac{\tau}{2}\}$ and
$$\det D^2G_{\frac{\tau}{4}}(z)<0\quad\text{for}\quad z\in\Big\{\frac12,\frac{1+\tau}{2},\pm\Big(\frac{1}{2}+\frac{\tau}{4}\Big)\Big\},$$
namely the nontrivial critical points $\pm(\frac{1}{2}+\frac{\tau}{4})$ are non-degenerate saddle points of $G_{\frac{\tau}{4}}(z)$. This proves the first assertion of (3).
Finally, we postpone the proof of the assertions (3-1)-(3-2) in Section \ref{section5}.
\end{proof}

\section{Generic non-degeneracy of critical points}
\label{section4}

In this section, we study the non-degeneracy of nontrivial critical points and prove Theorem \ref{main-thm-6}. Fix any $\tau$ and recall Hitchin's formula \eqref{513-1}. Define $f: \mathbb{R}^2\setminus\frac12\mathbb{Z}^2\to \mathbb{C}\cup\{\infty\}$ by
\begin{equation}\label{eq4-11}
f(r,s):=\wp (r+s\tau)+\frac{\wp ^{\prime }(r+s\tau)}{%
2(\zeta(r+s\tau)-r\eta_1-s\eta_2)}.
\end{equation}

Remark that in \cite{Hit1}, Hitchin's formula \eqref{513-1} was used as a function of $\tau$ with fixed $(r,s)$. Here we use Hitchin's formula \eqref{513-1} in a different way, i.e. we fix $\tau$ and consider it as a function of $(r,s)$.

Note that $f^{-1}(\{e_1,e_2,e_3,\infty\})$ is a closed subset of $\mathbb{R}^2$, so $$U:=\mathbb{R}^2\setminus\Big(\frac12\mathbb{Z}^2\cup f^{-1}(\{e_1,e_2,e_3,\infty\})\Big)$$ is an open subset of $\mathbb{R}^2$. We only need to consider the analytic map $f: U\to \mathbb{C}\setminus \{e_1,e_2,e_3\}$. 
Define
\begin{equation}\label{rmk-1-6}\Omega_{\tau}:=\left \{ p\in E_{\tau }\left \vert
\begin{array}{l}
\wp (p)=\wp (r+s\tau)+\frac{\wp ^{\prime }(r+s\tau)}{%
2(\zeta(r+s\tau)-r\eta_1-s\eta_2)}\text{ } \\
\text{for some }(r,s)\in \mathbb{R}^{2}\backslash \frac{1}{2}\mathbb{Z}^{2}%
\end{array}%
\right. \right \},\end{equation}
then $\Omega_{\tau}\setminus E_{\tau}[2]\neq \emptyset$, and Theorem \ref{thm-0B} implies that $G_p(z)$ has nontrivial critical points if and only if $p\in\Omega_{\tau}\setminus E_{\tau}[2]$. 
Clearly we have
\begin{equation}\label{Otau1}
\wp(\Omega_{\tau}\setminus E_{\tau}[2])=f(U).
\end{equation}
To prove Theorem \ref{main-thm-6}, we need to establish a precise formula of the Hessian $\det D^2G_p(q)$ in terms of $f$.

\begin{theorem}\label{thm-nonde}
Let $p\in \Omega_{\tau}\setminus E_{\tau}[2]$ and $q$ be a nontrivial critical point of $G_p(z)$. Write $q=r_0+s_0\tau$ with $(r_0, s_0)\in\mathbb{R}^2\setminus\frac12\mathbb{Z}^2$. Then $(r_0, s_0)\in U$, $\wp(p)=f(r_0, s_0)$ and there exists $c_{p,q}>0$ such that
\begin{equation}
\det D^2G_p(q)=-\frac{c_{p,q}}{16\pi^2\operatorname{Im}\tau}\det\begin{pmatrix}\frac{\partial \operatorname{Re}f}{\partial r} (r_0,s_0)& \frac{\partial \operatorname{Re}f}{\partial s}(r_0,s_0)\\\frac{\partial \operatorname{Im}f}{\partial r}(r_0,s_0) & \frac{\partial \operatorname{Im}f}{\partial s}(r_0,s_0)\end{pmatrix}.
\end{equation}
Consequently, $q$ is a degenerate nontrivial critical point of $G_p(z)$ if and only if
\[\det\begin{pmatrix}\frac{\partial \operatorname{Re}f}{\partial r} (r_0,s_0)& \frac{\partial \operatorname{Re}f}{\partial s}(r_0,s_0)\\\frac{\partial \operatorname{Im}f}{\partial r}(r_0,s_0) & \frac{\partial \operatorname{Im}f}{\partial s}(r_0,s_0)\end{pmatrix}=0.\]
\end{theorem}

\begin{proof}
For any $(r,s)\in U$, since $f(r,s)\in\mathbb{C}\setminus \{e_1,e_2,e_3\}$, there exists a unique pair $\pm p(r,s)\in E_{\tau}\setminus E_{\tau}[2]$ (i.e. $p(r,s)\neq -p(r,s)$ in $E_{\tau}$) such that $\wp(p(r,s))=f(r,s)$. Then we may assume that $p(r,s)$ is locally smooth as a function of $(r,s)\in U$. Recall \eqref{a+001p}-\eqref{513-1} that
$$\wp(p(r,s))=f(r,s)=\wp (r+s\tau)+\frac{\wp ^{\prime }(r+s\tau)}{%
2(\zeta(r+s\tau)-r\eta_1-s\eta_2)}$$
 is equivalent to
\begin{equation}\label{frsrs}
\zeta(r+s\tau+p(r,s))+\zeta(r+s\tau-p(r,s))-2(r\eta_1+s\eta_2)=0.
\end{equation}
Taking derivative with respect to $r$, we obtain
\begin{align*}-\big(\wp(r+s\tau&+p(r,s))+\wp(r+s\tau-p(r,s))+2\eta_1\big)\\
&-\big(\wp(r+s\tau+p(r,s))-\wp(r+s\tau-p(r,s))\big)\frac{\partial p}{\partial r}(r,s)=0,\end{align*}
and so
\begin{align}\label{frfr}
f_r&=\frac{\partial f}{\partial r}=\wp'(p(r,s))\frac{\partial p}{\partial r}\nonumber\\
&=\wp'(p(r,s))\frac{\wp(r+s\tau+p(r,s))+\wp(r+s\tau-p(r,s))+2\eta_1}{\wp(r+s\tau-p(r,s))-\wp(r+s\tau+p(r,s))}.
\end{align}
Here we have used $$\wp(r+s\tau-p(r,s))-\wp(r+s\tau+p(r,s))\neq 0.$$Indeed, if $\wp(r+s\tau-p(r,s))=\wp(r+s\tau+p(r,s))$, then $r+s\tau-p(r,s)=\pm (r+s\tau+p(r,s))$ in $E_{\tau}$, which implies $p(r,s)\in E_{\tau}[2]$ or $r+s\tau\in E_{\tau}[2]$, a contradiction.

Similarly, by taking derivative with respect to $s$ and using the Legendre relation $\tau\eta_1-\eta_2=2\pi i$, we have
\begin{align*}-\tau\big(\wp(r+s\tau&+p(r,s))+\wp(r+s\tau-p(r,s))+2\eta_1\big)+4\pi i\\
&-\big(\wp(r+s\tau+p(r,s))-\wp(r+s\tau-p(r,s))\big)\frac{\partial p}{\partial s}(r,s)=0,\end{align*}
and so
\begin{align*}f_s=\wp'(p(r,s))\bigg[&\tau\frac{\wp(r+s\tau+p(r,s))+\wp(r+s\tau-p(r,s))+2\eta_1}{\wp(r+s\tau-p(r,s))-\wp(r+s\tau+p(r,s))}\\
&\quad-\frac{4\pi i}{\wp(r+s\tau-p(r,s))-\wp(r+s\tau+p(r,s))}\bigg].\end{align*}
Therefore,
\begin{align}\label{frfs}
\frac{f_s}{f_r}(r,s)=\tau-\frac{4\pi i}{\wp(r+s\tau+p(r,s))+\wp(r+s\tau-p(r,s))+2\eta_1}.
\end{align}

Now $q=r_0+s_0\tau$ is a nontrivial critical point of $G_p(z)$, it follows from \eqref{a+1p} that
\[\zeta(r_0+s_0\tau+p)+\zeta(r_0+s_0\tau-p)-2(r_0\eta_1+s_0\tau_2)=0.\]
So we see from \eqref{frsrs} that $p=\pm p(r_0,s_0)$ and $\wp(p)=f(r_0,s_0)$. Since $p\notin E_{\tau}[2]$ implies $\wp(p)\notin\{e_1,e_2,e_3,\infty\}$, we obtain \begin{equation}\label{pfrs}(r_0,s_0)\in U\qquad\text{and}\qquad\wp(p)=f(r_0,s_0)\in f(U).\end{equation} Consequently, we deduce from \eqref{Gpdet}, \eqref{frfr} and \eqref{frfs} that
{\allowdisplaybreaks
\begin{align*}&\det D^2 G_p(q)\\
=&\frac{1}{4(\operatorname{Im}\tau)^2}\Big(1-\Big|\frac{\wp(q+p)+\wp(q-p)+2\eta_1}{2\pi}\operatorname{Im}\tau-1\Big|^2\Big)\\=&\frac{|\wp(q+p)+\wp(q-p)+2\eta_1|^2}{16\pi^2\operatorname{Im}\tau}\operatorname{Im}\left(\frac{4\pi i}{\wp(q+p)+\wp(q-p)+2\eta_1}-\tau\right)\\
=&-\frac{c_{p,q}}{16\pi^2\operatorname{Im}\tau}|f_r(r_0,s_0)|^2\operatorname{Im}\left(\frac{f_s}{f_r}(r_0,s_0)\right)\\
=&-\frac{c_{p,q}}{16\pi^2\operatorname{Im}\tau}\left(\operatorname{Re}f_r\cdot\operatorname{Im}f_s-\operatorname{Im}f_r\cdot\operatorname{Re}f_s\right)(r_0,s_0)\\
=&-\frac{c_{p,q}}{16\pi^2\operatorname{Im}\tau}\det\begin{pmatrix}\frac{\partial \operatorname{Re}f}{\partial r} (r_0,s_0)& \frac{\partial \operatorname{Re}f}{\partial s}(r_0,s_0)\\\frac{\partial \operatorname{Im}f}{\partial r}(r_0,s_0) & \frac{\partial \operatorname{Im}f}{\partial s}(r_0,s_0)\end{pmatrix},\end{align*}
}%
where
$$c_{p,q}:=\frac{|\wp(q-p)-\wp(q+p)|^2}{|\wp'(p)|^2}>0.$$
The proof is complete.
\end{proof}

Now we are ready to prove Theorem \ref{main-thm-6}. 

\begin{theorem}[=Theorem \ref{main-thm-6}]\label{main-thm-06}
Fix any $\tau$.
Recalling $\Omega_{\tau}$ defined in \eqref{rmk-1-6} and $\mathcal{T}_{\tau}$ defined in \eqref{O0tau}, 
we define
\begin{equation}\label{Otau}\mathcal{O}_{\tau}:=\Big\{p\in \Omega_{\tau}\setminus E_{\tau}[2]\;:\;\text{$G_p(z)$ has degenerate nontrivial critical points}\Big\}.\end{equation}
Then $\mathcal{O}_{\tau}$ is of Lebesgue measure zero, and $ \Omega_{\tau}\setminus (E_{\tau}[2]\cup \mathcal{O}_{\tau}\cup \mathcal{T}_{\tau})$ is open.  

Consequently, for almost all $p\in E_{\tau}\setminus E_{\tau}[2]$, all critical points of $G_p(z)$ are non-degenerate.
\end{theorem}

\begin{proof}{\bf Step 1.} We prove that $\mathcal{O}_\tau$ is of Lebesgue measure $0$.

Theorem \ref{thm-nonde} shows that $G_p(z)$ has degenerate nontrivial critical points if and only if $(\operatorname{Re}\wp(p), \operatorname{Im}\wp(p))$ is a singular value of the analytic map $(\operatorname{Re}f, \operatorname{Im}f): U\subset\mathbb{R}^2\to \mathbb{R}^2$, or equivalently, $\wp(p)$ is a singular value of the analytic map $f:U\to\mathbb C$.
So by applying Sard theorem, we see that
{\allowdisplaybreaks
\begin{align*}
&\wp(\mathcal{O}_\tau)\\
=&\{\wp(p)\in \wp(\Omega_{\tau}\setminus E_{\tau}[2]):\text{$G_p(z)$ has degenerate nontrivial critical points}\}\\
=&\{\wp(p)\in f(U)\;:\; \text{$\wp(p)$ is a singular value of $f:U\to\mathbb C$}\}
\end{align*}
}%
is of Lebesgue measure zero, so $\mathcal{O}_\tau$ is also of Lebesgue measure zero.
This proves that $\mathcal{O}_{\tau}\cup \mathcal{T}_{\tau}$ is of Lebesgue measure zero, and all critical points of $G_p(z)$ are non-degenerate for any $p\in E_{\tau}\setminus(E_{\tau}[2]\cup \mathcal{O}_{\tau}\cup \mathcal{T}_{\tau})$.

{\bf Step 2.}  We prove that $\Omega_{\tau}\setminus (E_{\tau}[2]\cup \mathcal{O}_{\tau}\cup \mathcal{T}_{\tau})$ is open.

Take any $p_0\in \Omega_{\tau}\setminus (E_{\tau}[2]\cup \mathcal{O}_{\tau}\cup \mathcal{T}_{\tau})$. Then $G_{p_0}(z)$ has nontrivial critical points and all critical points are non-degenerate. Denote by $N$ the number of critical points of $G_{p_0}(z)$, so $N\in\{6,8,10\}$.  Then by Lemma \ref{lemma50-0} below, there is $\varepsilon>0$ small such that for any $|p-p_0|<\varepsilon$, $G_{p}(z)$ has also exactly $N\geq 6$ critical points that are all non-degenerate.
Thus, $p\in \Omega_{\tau}\setminus (E_{\tau}[2]\cup \mathcal{O}_{\tau}\cup \mathcal{T}_{\tau})$ for any $|p-p_0|<\varepsilon$, namely 
$\Omega_{\tau}\setminus (E_{\tau}[2]\cup \mathcal{O}_{\tau}\cup \mathcal{T}_{\tau})$ is open.
The proof is complete.
\end{proof}

\begin{Lemma}\label{lemma50-0}
Fix $\tau$ and let $p_0\in E_{\tau}\setminus E_{\tau}[2]$. Suppose $G_{p_0}(z)$ has exactly $N\geq 4$ critical points which are all non-degenerate. Then there is $\varepsilon>0$ small such that for any $|p-p_0|<\varepsilon$, $G_{p}(z)$ has also exactly $N$ critical points that are all non-degenerate.

\end{Lemma}

\begin{proof} Denote by $q_1,\cdots,q_N$ the $N$ critical points of $G_{p_0}(z)$. Since they are all non-degenerate,  it follows from the implicit function theorem that there is $\varepsilon>0$ small such that for any $0<|p-p_0|<\varepsilon$, $G_p(z)$ has $N$ critical points $q_{p,1},\cdots, q_{p,N}$ (in small neighborhoods of each critical point $q_j$ of $G_{p_0}(z)$) and they are also non-degenerate. 
Assume by contradiction that there is $p_n\to p_0$ such that $G_{p_n}(z)$ has more than $N$ critical points, i.e. has another nontrivial critical point $q_{p_n,N+1}\in E_{\tau}\setminus \{q_{p_n,j}\}_{j=1}^{N}$. Up to a subsequence, we may assume 
$q_{p_n,N+1}\to q$. Write $q_{p_n,N+1}=r_n+s_n\tau$ with $\mathbb{R}^2\setminus\frac{1}{2}\mathbb{Z}^2\ni(r_n, s_n)\to (r,s)$ and $q=r+s\tau$. Then \eqref{a+1p} implies
$$\zeta(q_{p_n,N+1}+p_n)+\zeta(q_{p_n,N+1}-p_n)-2(r_n\eta_1+s_n\eta_2)=0.$$
This implies $q\neq \pm p_0$ in $E_{\tau}$ (otherwise $\zeta(q_{p_n,N+1}+p_n)+\zeta(q_{p_n,N+1}-p_n)\to \zeta(q+p_0)+\zeta(q-p_0)=\infty$, a contradiction). Then $q$ is a critical point of $G_{p_0}(z)$, so $q=q_j$ for some $1\leq j\leq N$, i.e. $q_{p_n,j}\to q_j$ and $q_{p_n,N+1}\to q_j$, a contradiction with the fact that $G_{p_n}(z)$ has a unique critical point $q_{p_n,j}$ in the small neighborhood of $q_j$ for $n$ large. Thus, by taking $\varepsilon>0$ smaller if necessary, $G_p(z)$ has exactly $N$ critical points that are all non-degenerate.
This completes the proof.
\end{proof}

\section{Counting the number of critical points}
\label{section5}

Thanks to Theorem \ref{main-thm-6} which says that all critical points of $G_p(z)$ are non-degenerate and so $0$ is a regular value of $\phi$ for almost all $p$, we can apply Corollary \ref{coro2-8} to study the possible distribution of the numbers of critical points of $G_p(z)$ when $p$ varies.
In this section, we apply this idea to prove Theorems \ref{thm-section5}, \ref{thm-section5-5} and Theorem \ref{main-thm-2} (3-1)-(3-2). 
\subsection{Proofs of Theorems \ref{thm-section5}, \ref{thm-section5-5} and Theorem \ref{main-thm-2} (3-1)-(3-2)}
\begin{Lemma}\label{lemma5-0}
Fix $\tau$ and let $\wp(p_0)\in\mathbb{C}\setminus (\{e_1,e_2,e_3\}\cup\cup_{k=0}^3\partial \mathcal{B}_k)$. Suppose $G_{p_0}(z)$ has exactly $4$ critical points. Then there is $\varepsilon>0$ small such that for any $|p-p_0|<\varepsilon$, $G_{p}(z)$ has also exactly $4$ critical points.
\end{Lemma}

\begin{proof}
Under the assumption on $p_0$, the $4$ critical points of $G_{p_0}(z)$ are precisely $\frac{\omega_k}{2}$'s that are all non-degenerate, so this lemma follows directly from Lemma \ref{lemma50-0}.
\end{proof}

\begin{proof}[Proof of Theorem \ref{thm-section5}]
Let $\Xi$ be a connected component of $\mathbb{C}\setminus (\{e_1,e_2,e_3\}\cup\cup_{k=0}^3\partial \mathcal{B}_k)$ and recall
$$m(\Xi)=\#\Big\{\frac{\omega_k}{2}\;: \;0\leq k\leq 3,\; \det D^2G_p\Big(\frac{\omega_k}{2}\Big)>0\;\text{for }p\in\wp^{-1}(\Xi)\Big\},$$ which is a constant independent of the choice of $p\in \wp^{-1}(\Xi)$ by Theorem \ref{main-thm-01}. By \eqref{eq330}, we have
$$m(\Xi)=\#\Big(D^+\cap\Big\{\frac{\omega_k}{2} \,:\, k=0,1,2,3\Big\}\Big).$$
Recall Theorem \ref{main-thm-6} that $0$ is a regular value of $\phi$ for almost all $\wp(p)\in \Xi$, so we can apply Corollary \ref{coro2-8} to obtain $m(\Xi)\leq 2$. 

(1) If $m(\Xi)=0$, then Corollary \ref{coro2-8}-(1) implies that $G_p(z)$ has exactly either $6$ or $10$ critical points for almost all $\wp(p)\in \Xi$, and so $G_p(z)$ has at least $6$ critical points for any $\wp(p)\in\Xi$ by using Lemma \ref{lemma5-0}.
Furthermore, when the number is $6$, since $m(\Xi)=0$ implies that $\frac{\omega_k}{2}$ are non-degenerate saddle points of $G_p(z)$ for all $k$, then the unique pair of nontrivial critical points are minimal points. 

(2) If $m(\Xi)=1$, then Corollary \ref{coro2-8}-(2) implies that $G_p(z)$ has exactly either $4$ or $8$ critical points for almost all $\wp(p)\in \Xi$.

(3) If $m(\Xi)=2$, then Corollary \ref{coro2-8}-(3) implies that $G_p(z)$ has exactly $6$ critical points for almost all $\wp(p)\in \Xi$. For such $\wp(p)$, since $m(\Xi)=2$, we see from the proof of Corollary \ref{coro2-8} that the unique pair of nontrivial critical points $\pm q$ satisfy $\pm q\in D^-$, i.e. $\det D^2G_p(\pm q)<0$ and so $\pm q$ are non-degenerate saddle points.

Again Lemma \ref{lemma5-0} implies that $G_p(z)$ has at least $6$ critical points for any $\wp(p)\in\Xi$. 
Recalling the analytic map $f: U\to\mathbb C$ defined in \eqref{eq4-11}, it follows from \eqref{Otau1} that $\Xi\subset f(U)$. Recall the proof of Theorem \ref{main-thm-6} in Section \ref{section4} that for any $\wp(p)\in \Xi$, all nontrivial critical points of $G_p(z)$ are non-degenerate (i.e. all critical points of $G_p(z)$ are non-degenerate) if and only if $\wp(p)$ is a regular value of $f$, so almost all $\wp(p)$ are regular values of $f$.

Now we claim that $G_p(z)$ has exactly $6$ critical points for any $\wp(p)\in\Xi$.

Assume by contradiction that there is $\wp(p_0)\in \Xi\subset f(U)$ such that $G_{p_0}(z)$ has $8$ or $10$ critical points.
Since $G_p(z)$ has exactly $6$ critical points for almost all $\wp(p)\in \Xi$, then $G_{p_0}(z)$ has a nontrivial critical point $q_0=r_0+s_0\tau$ (with $(r_0, s_0)\in U$) such that the following hold:
\begin{itemize}
\item[(P1)] There is a small $\varepsilon_0\in (0,\frac{1}{4})$ such that $r+s\tau$ is not a critical point of $G_{p_0}(z)$ for any $(r,s)\in \overline{B_{\varepsilon_0}}\setminus\{(r_0, s_0)\}$, where $$B_{\varepsilon_0}:=\big\{(r,s)\in\mathbb R^2 : |(r,s)-(r_0,s_0)|<\varepsilon_0\big\}\Subset U,$$
also satisifes $f(\overline{B_{\varepsilon_0}})\subset\Xi$.
In particular, $\wp(p_0)=f(r_0, s_0)$ and $\wp(p_0)\neq f(r,s)$ for any $(r,s)\in \overline{B_{\varepsilon_0}}\setminus\{(r_0, s_0)\}$.
\item[(P2)] There are regular values $\wp(p_n)\in \Xi$ of $f: U\to \mathbb C$ satisfying $p_n\to p_0$ such that $G_{p_n}(z)$ has no critical points in $\{r+s\tau : (r,s)\in \overline{B_{\varepsilon_0}}\}$, i.e. \begin{equation}\label{eq551}\wp(p_n)\notin f(\overline{B_{\varepsilon_0}}),\quad \forall n.\end{equation}
\end{itemize}

Note from $\wp(p_0)\notin f(\partial B_{\varepsilon_0})$ that the Brouwer degree
$$\deg(f, B_{\varepsilon_0}, \wp(p_0)):=\deg\big((\operatorname{Re}f, \operatorname{Im}f), B_{\varepsilon_0}, (\operatorname{Re}\wp(p_0), \operatorname{Im}\wp(p_0))\big)$$
is well-defined. Then there is a small $\delta>0$ such that for any $|p-p_0|<\delta$, $\wp(p)\notin f(\partial B_{\varepsilon_0})$, namely $\deg(f, B_{\varepsilon_0}, \wp(p))$ is well-defined and
$$\deg(f, B_{\varepsilon_0}, \wp(p))=\deg(f, B_{\varepsilon_0}, \wp(p_0)),\quad\forall |p-p_0|<\delta.$$
From here and \eqref{eq551}, we obtain 
\begin{equation}\label{fb0}
\deg(f, B_{\varepsilon_0}, \wp(p))=0, \quad\forall |p-p_0|<\delta.
\end{equation}

Recall $\Xi\subset f(U)$. Since $\wp(p)=f(r,s)$ is equivalent to that $\pm(r+s\tau)$ is a pair of nontrivial critical points of $G_p(z)$, we have $\#f^{-1}(\wp(p))\cap B_{\varepsilon_0}\leq 3$ for any $p$. Consequently, for any $B_\varepsilon:=\{(r,s)\in\mathbb R^2 : |(r,s)-(r_0,s_0)|<\varepsilon\}$ with $0<\varepsilon\leq \varepsilon_0$, the image $f(B_{\varepsilon})$ of the analytic map $f: B_{\varepsilon}\to \mathbb C$ has a positive Lesbegue measure (although $f(B_{\varepsilon})$ might not be open), so there is a regular value $\wp(p_1)\in f(B_{\varepsilon_0})\subset\Xi$ of $f$ satisfying $|p_1-p_0|<\delta$.
Then $G_{p_1}(z)$ has exactly $6$ critical points. 

On the other hand,  $\wp(p_1)\in f(B_{\varepsilon_0})$ implies the existence of $(r_1, s_1)\in B_{\varepsilon_0}$ such that $f(r_1, s_1)=\wp(p_1)$. Since $\wp(p_1)$ is a regular value of $f$, so $(r_1, s_1)$ contributes degree $\epsilon_1\in \{\pm 1\}$. But \eqref{fb0} says that $\deg(f, B_{\varepsilon_0}, \wp(p_1))=0$, so there is another $(r_2, s_2)\in B_{\varepsilon_0}$ such that $f(r_2, s_2)=\wp(p_1)$ and $(r_2, s_2)$ contributes the degree $-\epsilon_1$. This implies that $G_{p_1}(z)$ has at least two pairs of nontrivial critical points $\pm(r_1+s_1\tau)$ and $\pm(r_2+s_2\tau)$, or equivalently, at least $8$ critical points, a contradiction.

Therefore, $G_p(z)$ has exactly $6$ critical points for any $\wp(p)\in\Xi$.
This completes the proof. 
\end{proof}

To prove Theorem \ref{thm-section5-5}, first we prove the following weaker version.

\begin{Lemma}
\label{thm-section5-50}
There exist $\tau$ and a connected component $\Xi$ of $\mathbb{C}\setminus (\{e_1,e_2,e_3\}\cup\cup_{k=0}^3\partial \mathcal{B}_k)$ satisfying $m(\Xi)=0$ such that the following holds: for each $N\in \{6,10\}$, there is $\wp(p)\in\Xi$ such that $G_p(z;\tau)$ has exactly $N$ critical points that are all non-degenerate.
\end{Lemma}

\begin{proof} Clearly $\mathcal{B}_k$ depends on $\tau$, and we denote $\mathcal{B}_k=\mathcal{B}_k(\tau)$ in this proof.
Assume by contradiction that such $\tau$ and $\Xi$ does not exist. Then by Theorem \ref{thm-section5}-(1), we have that
\begin{equation} \label{eq: data0}
  \parbox{\dimexpr\linewidth-5em}{Given any connected component $\Xi$ of $\mathbb{C}\setminus (\{e_1,e_2,e_3\}\cup\cup_{k=0}^3\partial \mathcal{B}_k)$ satisfying $m(\Xi)=0$, then for almost all $\wp(p)\in\Xi$, the number of critical points of $G_p(z)$ is the same (this number is $6$ or $10$) and all critical points are non-degenerate.
  }
\end{equation}
Recall the proof of Theorem \ref{main-thm-2}-(2) in Section \ref{section3}: We can take $\tau_0$  
such that the Green function $G(z;\frac{\tau_0}{2})$ on the torus $E_{\frac{\tau_0}{2}}=\mathbb{C}/(\mathbb Z+\mathbb Z\frac{\tau_0}{2})$ has exactly $3$ critical points $\frac12, \frac{\tau_0}{4}, \frac{1}{2}+\frac{\tau_0}{4}$ with $\frac12$ being a degenerate critical point of $G(z;\frac{\tau_0}{2})$ and $\det D^2G(z;\frac{\tau_0}{2})<0$ for $z\in\{\frac{\tau_0}{4}, \frac{1}{2}+\frac{\tau_0}{4}\}$.
Then
 $G_{\frac{\tau_0}{4}}(z; \tau_0)$ has exactly $6$ critical points
$$ 0,\quad \frac12, \quad\frac{\tau_0}{2},\quad\frac{1+\tau_0}{2},\quad \pm\Big(\frac{1}{2}+\frac{\tau_0}{4}\Big),
$$
with $\pm(\frac{1}{2}+\frac{\tau_0}{4})$ being degenerate minimal points of $G_{\frac{\tau_0}{4}}(z;\tau_0)$ and
\begin{equation}\label{eq-data3}\det D^2G_{\frac{\tau_0}{4}}(z;\tau_0)<0\quad\text{for}\quad z\in\Big\{0, \frac12, \frac{\tau_0}{2},\frac{1+\tau_0}{2}\Big\}.\end{equation}
Let $\Xi_{\tau_0}$ be the connected component of $\mathbb{C}\setminus (\{e_1(\tau_0),e_2(\tau_0),e_3(\tau_0)\}\cup\cup_{k=0}^3\partial \mathcal{B}_k(\tau_0))$ such that $\wp(\frac{\tau_0}{4};\tau_0)\in \Xi_{\tau_0}$. Then $m(\Xi_{\tau_0})=0$ and by \eqref{eq: data0}, there are two possibilities.

{\bf Case 1.} $G_p(z;\tau_0)$ has exactly $6$ critical points that are all non-degenerate for almost all $\wp(p;\tau_0)\in\Xi_{\tau_0}$.

Fix $p_0$ close to $\frac{\tau_0}{4}$ such that $\wp(p_0;\tau_0)\in\Xi_{\tau_0}$ and $G_{p_0}(z;\tau_0)$ has exactly $6$ critical points that are all non-degenerate. Then by the implicit function theorem, the same proof as Lemma \ref{lemma50-0} implies the existence of small $\delta>0$ such that 
\begin{equation} \label{eq: data}
  \parbox{\dimexpr\linewidth-5em}{for any $|\tau-\tau_0|<\delta$ and $|p-p_0|<\delta$, $G_p(z;\tau)$ has exactly $6$ critical points that are all non-degenerate.
  }
\end{equation}
On the other hand, by Theorem \ref{thm-CKLW}, we have $\frac{\tau_0}{2}\in\mathcal{C}=\partial\Omega_5\cap\mathbb H$, so there is $\tau_n\to \tau_0$ such that the Green function $G(z;\frac{\tau_n}{2})$ on the torus $E_{\frac{\tau_n}{2}}$ has exactly $5$ critical points. Consequently, the proof of Theorem \ref{main-thm-2}-(1) implies that $G_{\frac{\tau_n}{4}}(z;\tau_n)$ has exactly $10$ critical points that are all non-degenerate and 
\begin{equation}\label{eq-data4}\det D^2G_{\frac{\tau_n}{4}}(z;\tau_n)<0\quad\text{for}\quad z\in\Big\{0, \frac12, \frac{\tau_n}{2},\frac{1+\tau_n}{2}\Big\}.\end{equation}
Let $\Xi_{\tau_n}$ be the connected component of $\mathbb{C}\setminus (\{e_1(\tau_n),e_2(\tau_n),e_3(\tau_n)\}\cup\cup_{k=0}^3\partial \mathcal{B}_k(\tau_n))$ such that $\wp(\frac{\tau_n}{4};\tau_n)\in \Xi_{\tau_n}$. Then $m(\Xi_{\tau_n})=0$ and by \eqref{eq: data0}, $G_p(z;\tau_n)$ has exactly $10$ critical points that are all non-degenerate for almost all $\wp(p;\tau_n)\in\Xi_{\tau_n}$. Since $\wp(p_0;\tau_0)\in\Xi_{\tau_0}\cap\Xi_{\tau_n}$ for $n$ large, there are large $n$ and $p_n$ satisfying $|\tau_n-\tau_0|<\delta, |p_n-p_0|<\delta$ and $\wp(p_n;\tau_n)\in \Xi_{\tau_n}$ such that $G_{p_n}(z;\tau_n)$ has exactly $10$ critical points, a contradiction with \eqref{eq: data}.

{\bf Case 2.} $G_p(z;\tau_0)$ has exactly $10$ critical points that are all non-degenerate for almost all $\wp(p;\tau_0)\in\Xi_{\tau_0}$.

Fix $p_0$ close to $\frac{\tau_0}{4}$ such that $\wp(p_0;\tau_0)\in\Xi_{\tau_0}$ and $G_{p_0}(z;\tau_0)$ has exactly $10$ critical points that are all non-degenerate. Again the same proof as Lemma \ref{lemma50-0} implies the existence of small $\delta>0$ such that 
\begin{equation} \label{eq: data1}
  \parbox{\dimexpr\linewidth-5em}{for any $|\tau-\tau_0|<\delta$ and $|p-p_0|<\delta$, $G_p(z;\tau)$ has exactly $10$ critical points that are all non-degenerate.
  }
\end{equation}
  On the other hand, by Theorem \ref{thm-CKLW}, we have $\frac{\tau_0}{2}\in\mathcal{C}=\partial\Omega_3\cap\mathbb H$, so there is $\tau_n\to \tau_0$ such that $\frac{\tau_n}{2}\in\Omega_3^\circ$, namely the Green function $G(z;\frac{\tau_n}{2})$ on the torus $E_{\frac{\tau_n}{2}}$ has exactly $3$ critical points that are all non-degenerate. Consequently, $G_{\frac{\tau_n}{4}}(z;\tau_n)$ has exactly $6$ critical points that are all non-degenerate. Furthermore, \eqref{eq-data3} and $\tau_n\to \tau_0$ together imply that \eqref{eq-data4} holds for $n$ large. Let $\Xi_{\tau_n}$ be the connected component of $\mathbb{C}\setminus (\{e_1(\tau_n),e_2(\tau_n),e_3(\tau_n)\}\cup\cup_{k=0}^3\partial \mathcal{B}_k(\tau_n))$ such that $\wp(\frac{\tau_n}{4};\tau_n)\in \Xi_{\tau_n}$. Then $m(\Xi_{\tau_n})=0$ and by \eqref{eq: data0}, $G_p(z;\tau_n)$ has exactly $6$ critical points that are all non-degenerate for almost all $\wp(p;\tau_n)\in\Xi_{\tau_n}$. Thus, there are large $n$ and $p_n$ satisfying $|\tau_n-\tau_0|<\delta, |p_n-p_0|<\delta$ and $\wp(p_n;\tau_n)\in \Xi_{\tau_n}$ such that $G_{p_n}(z;\tau_n)$ has exactly $6$ critical points, again a contradiction with \eqref{eq: data1}.
The proof is complete.
\end{proof}

\begin{proof}[Proof of Theorem \ref{thm-section5-5}] Lemma \ref{thm-section5-50} implies $\Omega_6\neq\emptyset$ and $\Omega_{10}\neq\emptyset$. Then it follows from Lemma \ref{lemma50-0} that $\Omega_6$ and $\Omega_{10}$ are both open subsets of $\Xi$. Furthermore,  
it follows from Theorem \ref{thm-section5}-(1) that $\Xi\subset\overline{\Omega_6}\cup\overline{\Omega_{10}}$. Since $\Xi$ is connected, we see that $\overline{\Omega_6}\cap\overline{\Omega_{10}}\cap\Xi\neq\emptyset$, which implies $\partial\Omega_6\cap\partial\Omega_{10}\cap \Xi\neq \emptyset$. Finally, take any $\wp(p)\in (\partial\Omega_6\cup\partial\Omega_{10})\cap\Xi$. If all critical points of $G_p(z)$ are non-degenerate, then $\wp(p)\in\Omega_6\cup\Omega_{10}$, a contradiction. Thus $G_p(z)$ has degenerate nontrivial critical points. The proof is complete.
\end{proof}

\begin{proof}[Proof of Theorem \ref{main-thm-2} (3-1)-(3-2)] As in the proof of Theorem \ref{main-thm-2}-(3), let us take $k=2$ for example. Recall the proof of Theorem \ref{main-thm-2}-(3) that we fix any $\tau$ such that $\frac{\tau}{4}$ is a degenerate critical point of $G(z;\frac{\tau}{2})$ and $\det D^2G(z;\frac{\tau}{2})<0$ for $z\in\{\frac{1}{2}, \frac{1}{2}+\frac{\tau}{4}\}$. Then $G_{\frac{\tau}{4}}(z)$ has exactly $6$ critical points
$$ 0,\quad \frac12, \quad\frac{\tau}{2},\quad\frac{1+\tau}{2},\quad \pm\Big(\frac{1}{2}+\frac{\tau}{4}\Big),
$$
with $\det D^2G_{\frac{\tau}{4}}(z)=0$ for $z\in \{0,\frac{\tau}{2}\}$ and
$$\det D^2G_{\frac{\tau}{4}}(z)<0\quad\text{for}\quad z\in\Big\{\frac12,\frac{1+\tau}{2},\pm\Big(\frac{1}{2}+\frac{\tau}{4}\Big)\Big\}.$$
Consequently, the implicit function theorem implies that for any $p$ close to $\frac{\tau}{4}$, $G_p(z)$ has exactly one pair of non-degenerate nontrivial critical points $\pm q_p$ (with $\det D^2G_p(\pm q_p)<0$) close to $\pm(\frac{1}{2}+\frac{\tau}{4})$, so $G_p(z)$ has at least $6$ critical points.

Note that $\wp(\frac{\tau}{4})\in \partial\mathcal{B}_0\cap\partial\mathcal{B}_2\cap(\mathbb C\setminus (\partial\mathcal{B}_1\cup\partial\mathcal{B}_3))$. Since \eqref{Bjk} says that $\partial\mathcal{B}_0\neq\partial\mathcal{B}_2$, there exist a connected component $\Xi_1$ of $\mathbb{C}\setminus (\{e_1,e_2,e_3\}\cup\cup_{k=0}^3\partial \mathcal{B}_k)$ and $j\in\{0,2\}$ such that $\wp(\frac{\tau}{4})\in\partial \Xi_1$ and 
$$\det D^2 G_p\left(\frac{\omega_j}{2}\right)>0,\quad \det D^2 G_p\left(\frac{\omega_{l}}{2}\right)<0\;\text{for }l\in\{1,3,2-l\}$$
holds for any $\wp(p)\in \Xi_1$, namely $m(\Xi_1)=1$. It follows from Theorem \ref{thm-section5}-(2) that $G_p(z)$ has exactly either $4$ or $8$ critical points for almost all $\wp(p)\in \Xi_1$. 
Since $G_p(z)$ has at least $6$ critical points for $p$ close to  $\frac{\tau}{4}$, so there exists $p$ close to $\frac{\tau}{4}$ satisfying $\wp(p)\in\Xi_1$ such that $G_p(z)$ has exactly $8$ critical points that are all non-degenerate. This proves (3-1).

 Similarly, there exists a connected component $\Xi_2$ of $\mathbb{C}\setminus (\{e_1,e_2,e_3\}\cup\cup_{k=0}^3\partial \mathcal{B}_k)$ satisfying $\wp(\frac{\tau}{4})\in\partial \Xi_2$ such that either 
\begin{equation}\label{cas-1}\det D^2 G_p\left(\frac{\omega_{l}}{2}\right)<0\;\text{for }l\in\{0,1,2,3\}\end{equation}
or
\begin{equation}\label{cas-2}\det D^2 G_p\left(\frac{\omega_l}{2}\right)>0\;\text{for }l\in\{0,2\},\quad \det D^2 G_p\left(\frac{\omega_{l}}{2}\right)<0\;\text{for }l\in\{1,3\}\end{equation}
holds for any $\wp(p)\in \Xi_2$. 
If \eqref{cas-1} holds, then $m(\Xi_2)=0$ and Theorem \ref{thm-section5}-(1) implies that $G_p(z)$ has exactly $6$ or $10$ critical points for almost all $\wp(p)\in \Xi_2$. Since $\det D^2G_p(\pm q_p)<0$ for $p$ close to $\frac{\tau}{4}$,  
 there exists $p$ close to $\frac{\tau}{4}$ satisfying $\wp(p)\in\Xi_2$ such that $\{\pm q_p, \frac{\omega_k}{2}, k=0,1,2,3\}\subset D^-$ and so
  $G_p(z)$ has exactly $10$ critical points that are all non-degenerate.  
  
 If \eqref{cas-2} holds, then $m(\Xi_2)=2$ and Theorem \ref{thm-section5}-(3) implies that $G_p(z)$ has exactly $6$ critical points for almost all $\wp(p)\in \Xi_2$, so 
 there exists $p$ close to $\frac{\tau}{4}$ satisfying $\wp(p)\in\Xi_2$ such that $G_p(z)$ has exactly $6$ critical points that are all non-degenerate. This proves (3-2).
  The proof is complete.
  \end{proof}
  
 \begin{Remark}\label{rmk-5-2}
Recall the proof of Theorem \ref{main-thm-2} (3-1)-(3-2) that $\wp(\frac{\tau}{4})\in \partial\mathcal{B}_0\cap\partial\mathcal{B}_2$. Suppose that $\partial\mathcal{B}_0$ is not tangent to $\partial\mathcal{B}_2$ at $\wp(\frac{\tau}{4})$ (we believe this assumption holds automatically), then there exists a connected component $\Xi_2$ satisfying $\wp(\frac{\tau}{4})\in\partial \Xi_2$ such that \eqref{cas-1} holds for any $\wp(p)\in \Xi_2$, and there also exists another connected component $\Xi_3$ satisfying $\wp(\frac{\tau}{4})\in\partial \Xi_3$ such that \eqref{cas-2} holds for any $\wp(p)\in \Xi_3$. Thus Theorem \ref{main-thm-2} (3-1)-(3-2) can be improved to that for each $N\in\{6,8,10\}$, there exists $p$ close to $\frac{\tau}{4}$ such that $G_p(z)$ has exactly $N$ critical points that are all non-degenerate.

 \end{Remark}


\subsection{Structure of critical points for $|\wp(p)|$ large}
Some consequences of Theorem \ref{thm-section5} are as follows, which shows how the geometry of the torus affects the structure of critical points of $G_p(z)$ when $|\wp(p)|$ is large.

\begin{corollary}\label{coro1-8}
Fix any $\tau$ such that $G(z)=G(z;\tau)$ has $5$ critical points, and define $$\Xi_\infty:=\mathbb{C}\setminus \big(\{e_1,e_2,e_3\}\cup\cup_{k=0}^3\overline{\mathcal{B}_k}\big).$$Then $G_p(z)$ has at least $6$ critical points for any $\wp(p)\in\Xi_\infty$, and $G_p(z)$ has exactly either $6$ or $10$ critical points for almost all $\wp(p)\in \Xi_\infty$. 
Furthermore, if the number is $6$, then the unique pair of nontrivial critical points are the minimal points of $G_p(z)$.
\end{corollary} 

\begin{proof}
Since $G(z)$ has $5$ critical points, Theorem \ref{thm-A0} says that $\det D^2G(\frac{\omega_k}{2})<0$ for $k=1,2,3$, so $\mathcal{B}_k$ are all open disks. This implies that $\Xi_\infty$ is unbounded. Moreover, it follows from Remark \ref{rmk3-11} and the proof of Lemma \ref{coro3-1} that
\begin{equation}\label{fcdc}\det D^2 G_p\Big(\frac{\omega_k}{2}\Big)>0\quad\Leftrightarrow\quad\wp(p)\in \mathcal{B}_k, \quad\text{for }k=0,1,2,3,\end{equation}
so
$m(\Xi_\infty)=0.$
Thus, Corollary \ref{coro1-8} follows from Theorem \ref{thm-section5}-(1).
\end{proof}

\begin{corollary}\label{coro1-82}
Fix any $\tau$ such that $G(z)=G(z;\tau)$ has exactly $3$ critical points that are all non-degenerate, and define $$\Xi_\infty:=\mathbb{C}\setminus \big(\{e_1,e_2,e_3\}\cup\cup_{k=0}^3\overline{\mathcal{B}_k}\big).$$Then $G_p(z)$ has exactly either $4$ or $8$ critical points for almost all $\wp(p)\in\Xi_\infty$.
\end{corollary} 

\begin{proof}
Again $\mathcal{B}_k$ are all open disks and $\Xi_\infty$ is unbounded. Recall Theorem \ref{thm-B} that if $|\wp(p)|$ is large enough, $G_p(z)$ has exactly $4$ critical points that are all non-degenerate, which together with Theorem \ref{thm-section5} implies the existence of $k_0\in\{1,2,3\}$ such that $\det D^2G_p(\frac{\omega_k}{2})<0$ for any $k\in\{0,1,2,3\}\setminus\{k_0\}$ and $\det D^2G_p(\frac{\omega_{k_0}}{2})>0$. This implies that $\det D^2G_p(\frac{\omega_k}{2})>0$ if and only if $\wp(p)\in\mathcal{B}_k$ for any $k\in\{0,1,2,3\}\setminus\{k_0\}$ and $\det D^2G_p(\frac{\omega_{k_0}}{2})>0$ if and only if $\wp(p)\in\mathbb{C}\setminus\overline{\mathcal{B}_{k_0}}$, so
$m(\Xi_\infty)=1.$
Thus, Corollary \ref{coro1-82} follows from Theorem \ref{thm-section5}-(2).
\end{proof}

\begin{corollary}\label{coro1-83}
Fix any $\tau$ such that $G(z)=G(z;\tau)$ has exactly $3$ critical points and one of them is degenerate (say $\frac{\omega_3}{2}$ for example). Define 
$$\Xi_{1,\infty}:=\mathbb{C}\setminus\big(\{e_1,e_2,e_3\}\cup\cup_{k=0}^3\overline{\mathcal{B}_k}\big),$$
$$\Xi_{2,\infty}:=\mathcal{B}_3\setminus\big(\{e_1,e_2,e_3\}\cup\cup_{k=0}^2\overline{\mathcal{B}_k}\big).$$
Then $\Xi_{j,\infty}$ are both unbounded, $G_p(z)$ has at least $6$ critical points for any $\wp(p)\in\Xi_{1,\infty}$, $G_p(z)$ has exactly either $6$ or $10$ critical points for almost all $\wp(p)\in \Xi_{1,\infty}$, and $G_p(z)$ has exactly either $4$ or $8$ critical points for almost all $\wp(p)\in\Xi_{2,\infty}$.
\end{corollary} 
\begin{proof}
Since $\frac{\omega_3}{2}$ is degenerate, Theorem \ref{thm-CKLW} says that $\det D^2G(\frac{\omega_k}{2})<0$ for $k=1,2$, so $\mathcal{B}_k$ is an open disk for $k=0,1,2$, but $\mathcal{B}_3$ is an open half plane, namely $\Xi_{1,\infty}$ and $\Xi_{2,\infty}$ are both unbounded. Moreover, it follows from Remark \ref{rmk3-11} and the proof of Lemma \ref{coro3-1} that \eqref{fcdc} holds, so $m(\Xi_{1,\infty})=0$ and $m(\Xi_{2,\infty})=1$. The proof is complete by using Theorem \ref{thm-section5}.
\end{proof}

\subsection{Two special examples}
Let us consider two special examples. The first example is to consider $\tau=e^{\pi i/3}=\frac12+\frac{\sqrt{3}}{2}i$. Note from Theorem \ref{thm-LW} that $G(z)$ has $5$ critical points for $\tau=\frac12+\frac{\sqrt{3}}{2}i$, so this is a special case covered in Corollary \ref{coro1-8}.

Let us recall the modular property. Given any $\bigl(\begin{smallmatrix}a & b\\
c & d\end{smallmatrix}\bigr)
\in SL(2,\mathbb{Z})$, it is well known that
\begin{equation}\label{II-301-00}
\wp \Big( \frac{z}{c\tau+d}; \frac{a\tau+b}{c\tau+d}\Big)  =\left(
c\tau+d\right)  ^{2}\wp (z;\tau),\end{equation}
\[
\zeta \Big(  \frac{z}{c\tau+d}; \frac{a\tau+b}{c\tau+d}\Big)
=\left(  c\tau+d\right)  \zeta( z;\tau),
\]
\[g_2\Big(\frac{a\tau+b}{c\tau+d}\Big)=(c\tau+d)^4g_2(\tau).\]
From here and (\ref{40-2}) we can obtain
\begin{equation}%
\begin{pmatrix}
\eta_{2}(\tfrac{a\tau+b}{c\tau+d})\\
\eta_{1}(\tfrac{a\tau+b}{c\tau+d})
\end{pmatrix}
=(c\tau+d)\begin{pmatrix}
a & b\\
c & d
\end{pmatrix}
\begin{pmatrix}
\eta_{2}(\tau)\\
\eta_{1}(\tau)
\end{pmatrix}
, \label{II-31-1-00}%
\end{equation}
and by (\ref{II-301-00}), we can easily derive
\[e_{1}\Big(\frac{a\tau+b}{c\tau+d}\Big)=\begin{cases}(c\tau+d)^2 e_{1}(\tau),\;\;\text{if $c$ even and $d$ odd}\\
(c\tau+d)^2 e_{2}(\tau),\;\;\text{if $c$ odd and $d$ even}\\
(c\tau+d)^2 e_{3}(\tau),\;\;\text{if $c$ odd and $d$ odd},\end{cases}\]
\begin{equation}\label{e3mo}e_{3}\Big(\frac{a\tau+b}{c\tau+d}\Big)=\begin{cases}(c\tau+d)^2 e_{1}(\tau),\;\;\text{if $a+c$ even and $b+d$ odd}\\
(c\tau+d)^2 e_{2}(\tau),\;\;\text{if $a+c$ odd and $b+d$ even}\\
(c\tau+d)^2 e_{3}(\tau),\;\;\text{if $a+c$ odd and $b+d$ odd}.\end{cases}\end{equation}
We recall the following expansion for $e_1(\tau)$ (See e.g. \cite[p.70]{Shimura}):
\begin{equation}
e_{1}(\tau)=\frac{2\pi^{2}}{3}+16\pi^{2}\sum_{k=1}^{\infty}a_{k}e^{2k\pi i \tau},\;\text{where}\; a_{k}=\sum_{d|k,d\text{ is odd}}d.
\label{ii-4}%
\end{equation}

Fix $\tau=e^{\pi i/3}=\frac12+\frac{\sqrt{3}}{2}i$. Then using the above modular properties with $\bigl(\begin{smallmatrix}a & b\\
c & d\end{smallmatrix}\bigr)=\bigl(\begin{smallmatrix}0 & 1\\
-1 & 1\end{smallmatrix}\bigr)$ and the Legendre relation $\tau\eta_1-\eta_2=2\pi i$, it follows from $\frac{1}{1-e^{\pi i/3}}=e^{\pi i/3}$ that
$$g_2(e^{\pi i/3})=0,\quad \eta_1(e^{\pi i/3})=\frac{2\pi}{\sqrt{3}}=\frac{\pi}{\operatorname{Im}\tau},$$
$$e_1(e^{\pi i/3})=\frac{2\pi^{2}}{3}+16\pi^{2}\sum_{k=1}^{\infty}(-1)^ka_{k}e^{-k\sqrt{3}\pi}\approx 1.8775\pi>\frac{2\pi}{\sqrt{3}},$$
$$e_2(e^{\pi i/3})=\overline{e_3(e^{\pi i/3})}=-e^{\pi i/3}e_1(e^{\pi i/3}).$$
From here, we obtain
$$\mathcal{B}_0=\Big\{z\in\mathbb{C}\; :\; |z|<\frac{2\pi}{\sqrt{3}}\Big\}.$$
Furthermore, 
\begin{align*}
\alpha_k=\frac{-e_k}{3e_k^2}=\frac{-1}{3e_k},
\quad
\beta_k=\frac{2\pi}{\sqrt{3}|3e_k^2|}=\frac{2\pi}{3\sqrt{3}e_1^2}<|\alpha_k|.
\end{align*}
Thus, 
\begin{align*}
e_k+\frac{\overline{\alpha_k}}{|\alpha_k|^2-\beta_k^2}=\begin{cases}-e_1\frac{6e_1^2+4\pi^2}{3e_1^2-4\pi^2}\approx-7.1816\pi & k=1,\\
e^{\pi i/3}e_1\frac{6e_1^2+4\pi^2}{3e_1^2-4\pi^2} & k=2,\\
e^{-\pi i/3}e_1\frac{6e_1^2+4\pi^2}{3e_1^2-4\pi^2} & k=3,\end{cases}
\end{align*}
\begin{align*}
\frac{\beta_k}{\left||\alpha_k|^2-\beta_k^2\right|}=\frac{6\sqrt{3}\pi e_1^2}{3e_1^2-4\pi^2}\approx 5.5715\pi.
\end{align*}
Then
$$\mathcal{B}_1=\bigg\{z\in\mathbb{C}\; :\; \bigg|z+e_1\frac{6e_1^2+4\pi^2}{3e_1^2-4\pi^2} \bigg|<\frac{6\sqrt{3}\pi e_1^2}{3e_1^2-4\pi^2}\bigg\},$$
$$\mathcal{B}_2=-e^{\pi i/3}\mathcal{B}_1,\quad \mathcal{B}_3=-e^{-\pi i/3}\mathcal{B}_1.$$
 Note that $\mathcal{B}_j\cap\mathcal{B}_k=\emptyset$ for any $j\neq k$, and $e_j\notin \mathcal{B}_k$ for any $j,k$.
The figure of these four disks is seen in Figure 1.
\begin{theorem}\label{main-thm-08}
Let $\tau=e^{\pi i/3}=\frac12+\frac{\sqrt{3}}{2}i$. Then the followings hold.
\begin{itemize}
\item[(1)] Denote$$\Xi_\infty:=\mathbb C\setminus\big(\{e_1,e_2,e_3\}\cup\cup_{k=0}^3\overline{\mathcal{B}_k}\big)\neq \emptyset,$$
then $G_p(z)$ has at least $6$ critical points for any $\wp(p)\in \Xi_\infty$, and $G_p(z)$ has exactly either $6$ or $10$ critical points for almost all $\wp(p)\in \Xi_\infty$.  Furthermore, if the number is $6$, then the unique pair of nontrivial critical points are the minimal points of $G_p(z)$.

\item[(2)] $G_p(z)$ has exactly either $4$ or $8$ critical points for almost all $\wp(p)\in \cup_{k=0}^3\mathcal{B}_k$. 
\end{itemize}
\end{theorem}

\begin{proof}
Theorem \ref{main-thm-08}-(1) follows directly from Corollary \ref{coro1-8}. Furthermore, 
for any $\wp(p)\in \cup_{k=0}^3\mathcal{B}_k$, it follows from \eqref{fcdc} that
$$\#\Big\{\frac{\omega_k}{2}\;: \;0\leq k\leq 3,\; \det D^2G_p\Big(\frac{\omega_k}{2}\Big)>0\Big\}\equiv 1.$$
Therefore, Theorem \ref{main-thm-08}-(2) follows directly from Theorem \ref{thm-section5}-(2).
\end{proof}

The second example is to consider $\tau=i$. Note from Theorem \ref{thm-LW} that $G(z)$ has only $3$ critical points that are all non-degenerate for $\tau=i$, so this is a special case covered in Corollary \ref{coro1-82}.
Fix $\tau=i$. Then using \eqref{II-31-1-00}-\eqref{ii-4} with $\bigl(\begin{smallmatrix}a & b\\
c & d\end{smallmatrix}\bigr)=\bigl(\begin{smallmatrix}0 & -1\\
1 & 0\end{smallmatrix}\bigr)$ and the Legendre relation $\tau\eta_1-\eta_2=2\pi i$, we obtain 
$$e_3(i)=0,\quad \eta_1(i)=\pi,$$
$$-e_2(i)=e_1(i)=\frac{2\pi^{2}}{3}+16\pi^{2}\sum_{k=1}^{\infty}a_{k}e^{-2k\pi}\approx 2.18844\pi,$$
so
$$e_3\in\mathcal{B}_0=\Big\{z\in\mathbb{C}\; :\; |z|<\pi\Big\}.$$
Furthermore, $g_2(i)=4(e_1(i)^2-e_2(i)e_3(i))=4e_1(i)^2$, which implies
\begin{align*}
\alpha_k=\frac{-e_k}{3e_k^2-e_1^2}=\begin{cases}\frac{-1}{2e_1} & k=1,\\
\frac{1}{2e_1} & k=2,\\
0 & k=3,\end{cases}
\end{align*}
\begin{align*}
\beta_k=\frac{\pi}{|3e_k^2-e_1^2|}=\begin{cases}\frac{\pi}{2e_1^2}<|\alpha_k| & k=1,2,\\
\frac{\pi}{e_1^2}>|\alpha_k| & k=3.\end{cases}
\end{align*}
Thus, 
\begin{align*}
e_k+\frac{\overline{\alpha_k}}{|\alpha_k|^2-\beta_k^2}=\begin{cases}-e_1\frac{e_1^2+\pi^2}{e_1^2-\pi^2}\approx-3.3435\pi & k=1,\\
e_1\frac{e_1^2+\pi^2}{e_1^2-\pi^2} & k=2,\\
0 & k=3,\end{cases}
\end{align*}
\begin{align*}
\frac{\beta_k}{\left||\alpha_k|^2-\beta_k^2\right|}=\begin{cases}\frac{2\pi e_1^2}{e_1^2-\pi^2}\approx 2.5278\pi & k=1,2,\\
\frac{e_1^2}{\pi}\approx 4.78927\pi & k=3.\end{cases}
\end{align*}
Then
$$e_3\in\mathcal{B}_0\subsetneqq\mathcal{B}_3=\bigg\{z\in\mathbb{C}\; :\; |z|<\frac{e_1^2}{\pi}\bigg\},$$
$$-\mathcal{B}_2=\mathcal{B}_1=\bigg\{z\in\mathbb{C}\; :\; \bigg|z+e_1\frac{e_1^2+\pi^2}{e_1^2-\pi^2} \bigg|<\frac{2\pi e_1^2}{e_1^2-\pi^2}\bigg\}.$$
 Note that $\mathcal{B}_1\cap\mathcal{B}_2=\emptyset$, $e_1\in \mathcal{B}_2\cap\mathcal{B}_3$ and $e_2\in \mathcal{B}_1\cap\mathcal{B}_3$.
The figure of these four disks is seen in Figure 2.
Then we have the following result.

\begin{theorem}\label{main-thm-8}
Let $\tau=i$. Then the followings hold.
\begin{itemize}
\item[(1)] Denote$$\Xi_1:=\mathbb{C}\setminus\cup_{k=0}^3 \overline{\mathcal{B}_k}\neq \emptyset,\quad \Xi_2:=\mathcal{B}_0\setminus(\{0\}\cup\cup_{k=1,2} \overline{\mathcal{B}_k})\neq \emptyset,$$
$$\Xi_3:=\mathcal{B}_2\cap\mathcal{B}_3\setminus(\{e_1\}\cup\overline{\mathcal{B}_0})\neq \emptyset,\quad\widetilde{\Xi}_3:=\mathcal{B}_1\cap\mathcal{B}_3\setminus(\{e_2\}\cup\overline{\mathcal{B}_0})=-\Xi_3,$$
then $G_p(z)$ has exactly either $4$ or $8$ critical points for almost all $\wp(p)\in \Xi_1\cup \Xi_2\cup \Xi_3\cup\widetilde{\Xi}_3$.

\item[(2)] Denote $$\Xi_4:=\mathcal{B}_2\setminus\overline{\mathcal{B}_3}\neq \emptyset,\quad\widetilde{\Xi}_4:=\mathcal{B}_1\setminus \overline{\mathcal{B}_3}=-\Xi_4,$$
$$\Xi_5:=\mathcal{B}_0\cap\mathcal{B}_2\neq \emptyset,\quad\widetilde\Xi_5:=\mathcal{B}_0\cap\mathcal{B}_1=-\Xi_5,$$
then $G_p(z)$ has exactly $6$ critical points for any $\wp(p)\in \Xi_4\cup\widetilde\Xi_4\cup\Xi_5\cup\widetilde\Xi_5$, and the unique pair of nontrivial critical points are non-degenerate saddle points of $G_p(z)$ for almost all $\wp(p)\in \Xi_4\cup\widetilde\Xi_4\cup\Xi_5\cup\widetilde\Xi_5$.

\item[(3)]  Denote $$\Xi_6:=\mathcal{B}_3\setminus\cup_{k=0,1,2} \overline{\mathcal{B}_k}\neq \emptyset,$$ then $G_p(z)$ has at least $6$ critical points for any $\wp(p)\in \Xi_6$, and $G_p(z)$ has exactly either $6$ or $10$ critical points for almost all $\wp(p)\in \Xi_6$.  Furthermore, if the number is $6$, then the unique pair of nontrivial critical points are the minimal points of $G_p(z)$.
\end{itemize}
\end{theorem}

\begin{proof}
By \eqref{eq330} with $k=0$, \eqref{eq331} with $k=1,2$, and \eqref{eq332} with $k=3$, we have
$$\det D^2 G_p\Big(\frac{\omega_k}{2}\Big)>0\quad\Leftrightarrow\quad\wp(p)\in \mathcal{B}_k, \quad\text{for }k=0,1,2,$$
$$\det D^2 G_p\Big(\frac{\omega_3}{2}\Big)>0\quad\Leftrightarrow\quad\wp(p)\in \mathbb C\setminus\overline{\mathcal{B}_3}.$$
Then there are the following cases.

(1). For any $\wp(p)\in \Xi_1\cup \Xi_2\cup \Xi_3\cup\widetilde{\Xi}_3$, it is easy to see that
$$\#\Big\{\frac{\omega_k}{2}\;: \;0\leq k\leq 3,\; \det D^2G_p\Big(\frac{\omega_k}{2}\Big)>0\Big\}\equiv 1.$$
Therefore, Theorem \ref{main-thm-8}-(1) follows directly from Theorem \ref{thm-section5}-(2).

(2). For any $\wp(p)\in \Xi_4\cup\widetilde\Xi_4\cup\Xi_5\cup\widetilde\Xi_5$, 
it is easy to see that
$$\#\Big\{\frac{\omega_k}{2}\;: \;0\leq k\leq 3,\; \det D^2G_p\Big(\frac{\omega_k}{2}\Big)>0\Big\}\equiv 2.$$
Therefore, Theorem \ref{main-thm-8}-(2) follows directly from Theorem \ref{thm-section5}-(3).

(3) For any $\wp(p)\in \Xi_6$, 
it is easy to see that
$$\#\Big\{\frac{\omega_k}{2}\;: \;0\leq k\leq 3,\; \det D^2G_p\Big(\frac{\omega_k}{2}\Big)>0\Big\}\equiv 0.$$
Therefore, Theorem \ref{main-thm-8}-(3) follows directly from Theorem \ref{thm-section5}-(1).
\end{proof}

\section{Connection with the curvature equation}
\label{section6}

In this section,
we study the deep connection between $G_p(z)$ and the curvature equation \eqref{mean}. First, we prove Theorem \ref{main-thm-3} by following the approach from \cite{CLW,LW} where the equation $\Delta u+e^u=8\pi n\delta_0$ on $E_{\tau}$ was studied.
We divide the proof into two lemmas.
\begin{Lemma}\label{lemma4-1}
Let $p\in E_{\tau}\setminus E_{\tau}[2]$. Suppose $G_p(z)$ has a pair of nontrivial critical points $\pm q$. Then the curvature equation \eqref{mean} has a one-parameter scaling family of solutions $u_{\beta}(z)$, where $\beta>0$ is arbitrary. Furthermore, $u_{\beta}(z)=u_\beta(-z)$ if and only if $\beta=1$, and $u_{\beta}(z)$ blows up at $q$ as $\beta\to+\infty$, and $u_{\beta}(z)$ blows up at $-q$ as $\beta\to0$.
\end{Lemma}

Here we say $u_{\beta}(z)$ blows up at $q$, or equivalently, $q$ is a blowup point of $u_{\beta}(z)$, as $\beta\to+\infty$, if there exists $z_\beta\to q$ such that $u_{\beta}(z_{\beta})\to+\infty$.

\begin{proof}
Write $q=r+s\tau$ with $(r,s)\in\mathbb{R}^2\setminus\frac12\mathbb Z^2$. Recall \eqref{a+1p} that
\begin{equation}\label{a+4p}\zeta(q+p)+\zeta(q-p)-2(r\eta_1+s\eta_2)=0.\end{equation}
Let $\sigma(z)=\sigma(z;\tau):=\exp(\int^{z}\zeta(\xi;\tau)d\xi)$ be the Weierstrass sigma function, which is an odd
entire function with simple zeros at $\Lambda_{\tau}$, and satisfies the following transformation law
\begin{equation}\label{40-3}
\sigma(z+\omega_{k})=-e^{\eta_{k}(z+\omega_{k}/2)}\sigma(z),\quad k=1,2.
\end{equation}
Define
\begin{align}\label{4-00}
f(z):=&\exp((\zeta(q+p)+\zeta(q-p))z)\frac{\sigma(z-q)}{\sigma(z+q)}\nonumber\\
=&\exp(2(r\eta_1+s\eta_2)z)\frac{\sigma(z-q)}{\sigma(z+q)}.
\end{align}
Then it follows from \eqref{40-3} and the Legendre relation $\tau\eta_1-\eta_2=2\pi i$ that
\begin{align}\label{4-01}
f(z+1)=e^{-4\pi i s}f(z),\quad f(z+\tau)=e^{4\pi i r}f(z).
\end{align}
Define $u_\beta(z)$ via the Liouville formula
\begin{align}\label{liou}
u_\beta(z):=\log\frac{8|\beta f'(z)|^2}{(1+|\beta f(z)|^2)^2},\quad \forall \beta>0.
\end{align}
Then \eqref{4-01} and $(r,s)\in\mathbb{R}^2$ imply that $u_\beta(z)$ is doubly-periodic and hence well-defined in $E_{\tau}$.
Notice from \eqref{4-00} that
$$\frac{f'(z)}{f(z)}=\zeta(q+p)+\zeta(q-p)+\zeta(z-q)-\zeta(z+q),$$
so $f'(z)$ has exactly two simple zeros $\pm p$ and then
$$u_\beta(z)=2\log|z\mp p|+O(1)\quad\text{near }\pm p.$$
Furthermore, it is easy to see from \eqref{liou} that $u_\beta(z)\neq \infty$ for any $z\in E_{\tau}\setminus\{\pm p\}$ and a direct computation implies
\begin{align}
\Delta u_\beta+ e^{u_{\beta}}=0\quad\text{in}\quad E_{\tau}\setminus\{\pm p\}.
\end{align}
This proves that $u_\beta(z)$ is a one-parameter scaling family of solutions of \eqref{mean}. Furthermore, since $f(-z)=\frac{1}{f(z)}$ but $\beta f(-z)\neq\frac{1}{\beta f(z)}$ for $\beta\neq 1$, we see that $u_{\beta}(z)=u_\beta(-z)$ if and only if $\beta=1$.

Finally, by $f(q)=0$ and $f'(q)\neq 0$ we know that $u_\beta(q)\to+\infty$ as $\beta\to+\infty$, so $q$ is a blowup point of $u_\beta$ as $\beta\to+\infty$. Since $\int_{E_{\tau}}e^{u_\beta}\equiv 8\pi$, it follows from \cite{BT} that $q$ is the unique blowup point. Similarly, we know that $-q$ is the unique blowup point of $u_{\beta}(z)$ as $\beta\to 0$. The proof is complete.
\end{proof}

\begin{Lemma}\label{lemma4-2}
Let $p\in E_{\tau}\setminus E_{\tau}[2]$. Suppose the curvature equation \eqref{mean} has a solution $u(z)$, then it belongs to a one-parameter scaling family of solutions $u_{\beta}(z)$. Furthermore, $u_{\beta}(z)$ blows up at some $q\notin E_{\tau}[2]$ as $\beta\to+\infty$. In particular, $\pm q$ is a pair of nontrivial critical points of $G_p(z)$.
\end{Lemma}

\begin{proof} 
Let $u(z)$ be a solution of (\ref{mean}). Then
the classical Liouville theorem (cf. \cite[Section 1.1]{CLW}) says that there
is a local meromorphic function $f(z)$ away from $\{\pm p\}$ such that%
\begin{equation}
u(z)=\log \frac{8|f^{\prime}(z)|^{2}}{(1+|f(z)|^{2})^{2}}. \label{502}%
\end{equation}
This $f(z)$ is called a developing map. By differentiating (\ref{502}), we
have
\begin{equation}
u_{zz}-\frac{1}{2}u_{z}^{2}= \{ f;z \}:=\left(  \frac{f^{\prime \prime}%
}{f^{\prime}}\right)  ^{\prime}-\frac{1}{2}\left(  \frac{f^{\prime \prime}%
}{f^{\prime}}\right)  ^{2}. \label{new22}%
\end{equation}
Conventionally, the RHS of this identity is called the Schwarzian derivative
of $f(z)$, denoted by $\{ f;z \}$. It was proved in \cite[Section 3]{CKL-2025} that there are constants $A$ and $B$ such that $u_{zz}-\frac{1}{2}u_{z}^{2}$ is an elliptic function that can be expressed as
\begin{align}
u_{zz}-\frac{1}{2}u_{z}^{2}=-2\Big[
\frac{3}{4}
(\wp(z+p)+\wp(z-p))+A(\zeta(z+p)-\zeta(z-p))+B
\Big]. \label{cc-1}%
\end{align}
Then a classical result in complex analysis says that there are linearly independent solutions $y_1(z), y_2(z)$ of the second order linear ODE
\begin{align}\label{GLE}
y''(z)=\Big[&
\frac{3}{4}
(\wp(z+p)+\wp(z-p))\\
&+A(\zeta(z+p)-\zeta(z-p))+B
\Big]y(z),\quad z\in\mathbb{C}\nonumber
\end{align}
 such that
\begin{equation}\label{fy}f(z)=\frac{y_1(z)}{y_2(z)}.\end{equation}
Let $\gamma_j^*y(z)$ denote the analytic continuation of $y(z)$ along the fundamental cycle $z\to z+\omega_j$ of $E_{\tau}$, $j=1,2$. Then there are two monodromy matrices $N_1, N_2$ such that
\begin{equation}\label{monon1n2}\gamma_j^*\begin{pmatrix}
y_{1}(z)\\
y_{2}(z)
\end{pmatrix}
=N_j
\begin{pmatrix}
y_{1}(z)\\
y_{2}(z)
\end{pmatrix},\quad j=1,2.\end{equation}
Define the Wronskian
$W:=y_{1}'(z)y_2(z)-y_1(z)y_2'(z).$
Then $W$ is a nonzero constant, and inserting (\ref{fy}) into (\ref{502}) leads to
\[2\sqrt{2}W e^{-\frac{1}{2}u(z)}=|y_1(z)|^2+|y_2(z)|^2.\]
Since $u(z)$ is single-valued and doubly periodic, so $e^{-\frac{1}{2}u(z)}$ is invariant under analytic continuation along $ \gamma_j$, which implies
$$(\overline{y_1(z)}, \overline{y_2(z)})\overline{N_j}^TN_j\begin{pmatrix}y_1(z)\\ y_2(z)\end{pmatrix}=|y_1(z)|^2+|y_2(z)|^2.$$
Thus $\overline{N_j}^TN_j=I_2$, where $I_2$ is the identity matrix. This proves that $N_j\in SU(2)$ are unitary matrices, namely the monodromy group of \eqref{GLE} is a subgroup of the unitary group $SU(2)$. 

On the other hand, the monodromy theory of \eqref{GLE} has been studied in \cite{CKL-PAMQ}. In particular,  togethwer with $N_j\in SU(2)$, 
 it follows from \cite[Section 2]{CKL-PAMQ} that there exists $(r,s)\in\mathbb{R}^2\setminus\frac{1}{2}\mathbb{Z}^2$ such that by denoting $q=r+s\tau$, the following facts hold:
\begin{itemize}
\item[(1)] $q\neq \pm p$ in $E_{\tau}$ and
\begin{align}\label{q+p}
\zeta(q+p)+\zeta(q-p)-2(r\eta_1+s\eta_2)=0.\end{align}
\item[(2)] Define
$$y_{\pm q}(z):=\exp(\pm (r\eta_1+s\eta_2)z)\frac{\sigma(z\mp q)}{\sqrt{\sigma(z-p)\sigma(z+p)}},$$
then $(y_{q}(z), y_{-q}(z))$ is a basis of solutions of \eqref{GLE}, and 
$$\gamma_1^*\begin{pmatrix}
y_{q}(z)\\
y_{-q}(z)
\end{pmatrix}
=\begin{pmatrix}
e^{-2\pi is}&\\
&e^{2\pi is}
\end{pmatrix}
\begin{pmatrix}
y_{q}(z)\\
y_{-q}(z)
\end{pmatrix},$$
$$\gamma_2^*\begin{pmatrix}
y_{q}(z)\\
y_{-q}(z)
\end{pmatrix}
=\begin{pmatrix}
e^{2\pi ir}&\\
&e^{-2\pi ir}
\end{pmatrix}
\begin{pmatrix}
y_{q}(z)\\
y_{-q}(z)
\end{pmatrix}.$$
\end{itemize}
Remark that \eqref{q+p} already implies that $q$ is a nontrivial critical point of $G_p(z)$. Define
$$\tilde{f}(z):=\frac{y_{q}(z)}{y_{-q}(z)}=\exp(2(r\eta_1+s\eta_2)z)\frac{\sigma(z-q)}{\sigma(z+q)},$$
and 
\begin{align}\label{liou0}
u_\beta(z):=\log\frac{8|\beta \tilde{f}'(z)|^2}{(1+|\beta \tilde{f}(z)|^2)^2},\quad \forall \beta>0.
\end{align}
Then by Lemma \ref{lemma4-1}, it suffices to prove the existence of $\beta_0>0$ such that $u(z)=u_{\beta_0}(z)$.

Indeed, there is an invertible matrix $P=\bigl(\begin{smallmatrix}a & b\\
c & d\end{smallmatrix}\bigr)$ such that 
$$\begin{pmatrix}
y_{1}(z)\\
y_{2}(z)
\end{pmatrix}=\begin{pmatrix}
a&b\\
c&d
\end{pmatrix}\begin{pmatrix}
y_{q}(z)\\
y_{-q}(z)
\end{pmatrix},$$
so
\begin{equation}\label{fft}f(z)=\frac{ay_{q}(z)+by_{-q}(z)}{cy_{q}(z)+dy_{-q}(z)}=\frac{a\tilde{f}(z)+b}{c\tilde{f}(z)+d},\end{equation}
and
$$P\begin{pmatrix}
e^{-2\pi is}&\\
&e^{2\pi is}
\end{pmatrix}P^{-1}=N_1,\quad P\begin{pmatrix}
e^{2\pi ir}&\\
&e^{-2\pi ir}
\end{pmatrix}P^{-1}=N_2.$$
Since $(r,s)\in\mathbb{R}^2\setminus \frac{1}{2}\mathbb{Z}^2$, without loss of generality, we may assume $r\in \mathbb{R}\setminus\frac12\mathbb Z$, i.e. $e^{2\pi ir}\neq \pm 1$. By $N_2\in SU(2)$, i.e. $\overline{N_2}^TN_2=I_2$, we obtain from $P\bigl(\begin{smallmatrix}e^{2\pi ir} & \\
 & e^{-2\pi ir}\end{smallmatrix}\bigr)=N_2P$ that
$$\overline{P}^TP\begin{pmatrix}
e^{2\pi ir}&\\
&e^{-2\pi ir}
\end{pmatrix}=\begin{pmatrix}
e^{2\pi ir}&\\
&e^{-2\pi ir}
\end{pmatrix}\overline{P}^TP,$$
which implies that $\overline{P}^TP$ is a diagonal matrix, i.e. $a\bar{b}+c\bar{d}=0$, so there is $\alpha\neq 0$ such that $(a,c)=\alpha (-\bar{d}, \bar{b})$. Consequently,
$$|ad-bc|=|\alpha| (|b|^2+|d|^2),\quad |a|^2+|c|^2=|\alpha|^2(|b|^2+|d|^2).$$
Let $\beta_0=|\alpha|>0$. From here and by inserting \eqref{fft} into \eqref{502}, we easily obtain that
{\allowdisplaybreaks
\begin{align*}
u(z)=&\log\frac{8|ad-bc|^2|\tilde{f}'(z)|^2}{(|a\tilde{f}(z)+b|^2+|c\tilde{f}(z)+d|^2)^2}\\
=&\log\frac{8|ad-bc|^2|\tilde{f}'(z)|^2}{(|b|^2+|d|^2+(|a|^2+|c|^2)|\tilde{f}(z)|^2)^2}\\
=&\log\frac{8|\beta_0 \tilde{f}'(z)|^2}{(1+|\beta_0 \tilde{f}(z)|^2)^2}=u_{\beta_0}(z).
\end{align*}
}%
The proof is complete.
\end{proof}

\begin{proof}[Proof of Theorem \ref{main-thm-3}]
Lemmas \ref{lemma4-1} and \ref{lemma4-2} together imply the one-to-one correspondence between pairs of nontrivial critical points of $G_p(z)$ and one-parameter scaling families of solutions (or equivalently, even solutions) of \eqref{mean}. 
Consequently, the rest assertions follows directly from those results for $G_p(z)$.
\end{proof}

Finally, we give the proof of Theorem \ref{main-thm-10}.

\begin{proof}[Proof of Theorem \ref{main-thm-10}]
Notice that if \eqref{mean} has solutions for $0\neq p\to 0$ (such as for those $E_{\tau}$ such that $G(z;\tau)$ has $5$ critical points by Theorem \ref{main-thm-3}-(2)), then by Lemma \ref{lemma4-2}, \eqref{mean} has a solution $u_p(z)$ satisfying $\max_z u_p(z)> \frac{1}{|p|}$  for each $p$, so $u_p(z)$ blows up as $p\to 0$. This indicates that Theorem \ref{main-thm-10} (1)-(2) are not true if we delete the assumption ``even''.

Fix any $E_\tau$. There are three cases.

{\bf Case 1.} $G(z;\tau)$ has exactly $3$ critical points $\frac{\omega_k}{2}$'s that are all non-degenerate (such as $\tau\in i\mathbb{R}_{>0}$).

Then Theorem \ref{thm-B} implies that for $|p|>0$ small, $G_p(z)$ has no nontrivial critical points, or equivalently, \eqref{mean} has no solutions.

{\bf Case 2.} $G(z;\tau)$ has $5$ critical points.

Then by Theorem \ref{thm-5c}, for $|p|>0$ small, $G_p(z)$ has a unique pair of nontrivial critical points $\pm q_p$, which converge to the unique pair of nontrivial critical points $\pm q$ of $G(z)$ as $p\to 0$. Consequently, it follows from Theorem \ref{main-thm-3} that the even solution $u_p(z)$ of \eqref{mean} does exist and is unique for $|p|$ small, and it follows from the proof of Lemma \ref{lemma4-1} that
$$u_p(z)=\log\frac{8|f_p'(z)|^2}{(1+|f_p(z)|^2)^2},$$
where 
$$f_p(z)=\exp((\zeta(q_p+p)+\zeta(q_p-p))z)\frac{\sigma(z-q_p)}{\sigma(z+q_p)}.$$
Now $q_p\to q\notin E_{\tau}[2]$ as $p\to 0$ implies that $f_p(z)\to f_0(z)=\exp(2\zeta(q)z)\frac{\sigma(z-q)}{\sigma(z+q)}$ and $u_p(z)\to u(z)=\log\frac{8|f_0'(z)|^2}{(1+|f_0(z)|^2)^2}$, where $u(z)$ is an even solution of \eqref{mfe-8pi} (Note that $q\notin E_{\tau}[2]$ is crucial here, because it follows from \eqref{40-3} that $f_0(z)\equiv -1$ and $u_p(z)\to-\infty$ locally if $q=\frac{\omega_k}{2}$ for some $k=1,2,3$).

{\bf Case 3.}
$G(z;\tau)$ has exactly $3$ critical points $\frac{\omega_k}{2}$'s and some of them is degenerate.

Say $\frac{\omega_3}{2}$ is degenerate for example. Then by Corollary \ref{coro1-83}, 
 we see that $G_p(z)$ has at least $6$ critical points for any $\wp(p)\in\Xi_{1,\infty}$ and $\Xi_{1,\infty}$ is unbounded. This, together with Theorem \ref{main-thm-3}, implies that \eqref{mean} has even solutions $u_p(z)$ for any $\wp(p)\in\Xi_{1,\infty}$. Since it was proved in \cite{LW} that \eqref{mfe-8pi} has no solutions when $G(z;\tau)$ has no nontrivial critical points, we conclude that $u_p(z)$ must blow up as $p\to 0$. 
The proof is complete.
\end{proof}

\subsection*{Acknowledgements} The authors thank A. Eremenko for his interest and comments. Z. Chen was supported by National Key R\&D Program of China (No. 2023YFA1010002) and NSFC (No. 12222109). 
E. Fu was supported by NSFC (No. 12401188) and BIMSA Start-up Research Fund.

\end{document}